\documentclass[11pt, a4paper]{article}

\usepackage[T1]{fontenc}
\usepackage{lmodern}

\usepackage[a4paper, margin=1in]{geometry}
\usepackage{setspace}
\usepackage{amssymb, amsfonts, amsthm, amsmath, mathtools}
\usepackage{nicefrac}
\usepackage{bm, bbm}
\usepackage{enumitem}
\usepackage{booktabs, multirow, multicol, makecell} 
\usepackage{comment}
\usepackage{hyperref}

\newtheoremstyle{proclaim}
  {3pt}        
  {3pt}        
  {\slshape}  
  {}        
  {\bfseries} 
  {.}       
  { }    
  {}        

\theoremstyle{proclaim}

\newtheorem{lemma}{Lemma}
\newtheorem{proposition}{Proposition}
\newtheorem{theorem}{Theorem}

\theoremstyle{definition}

\newtheorem{example}{Example}

\theoremstyle{remark}
\newtheorem*{remark*}{Remark}
\newtheorem{remark}{Remark}

\usepackage{appendix}

\newcommand{\blank}{\kern.15em\cdot\kern.15em}
\makeatletter
\newcommand*{\transpose}{{\mathpalette\@transpose{}}}
\newcommand*{\@transpose}[2]{\raisebox{\depth}{$\m@th#1\kern0.0556em\intercal$}}
\makeatother
\DeclareMathOperator{\Indicator}{\mathbbm{1}}
\newcommand{\awDistance}{d\kern-.12em l} 
\newcommand{\gtrasymp}{\mathrel{\raisebox{2.25pt}{$>$}\kern-0.75em\raisebox{-3.25pt}{$\smallfrown$}}}

\newcommand{\qqquad}{\quad\quad\quad}
\newcommand{\tspace}{\kern0.0556em} 
\newcommand{\ttspace}{\tspace\tspace}
\newcommand{\ntspace}{\kern-0.0556em}

\newcommand{\reals}{\mathbb{R}}

\DeclareMathAlphabet{\mathsl}{T1}{cmr}{m}{sl}
\newcommand{\drv}{\mathsl{d}}

\newcommand{\defeq}{\coloneq}

\DeclareMathOperator*{\minimize}{minimize}
\DeclareMathOperator*{\maximize}{maximize}

\newcommand{\subjectTo}{\text{subject to}}

\newcommand{\where}{\text{where}}

\usepackage{microtype}

\usepackage{csquotes} 
\usepackage[british]{babel} 
\newcommand{\rxi}{\bm{\xi}}

\newcommand{\Expt}{\mathbb{E}} 
\newcommand{\Prob}{\mathbb{P}} 
\newcommand{\pmProb}{\mathbbm{1}} 
\newcommand{\given}{\kern.1em\vert\kern.1em}
\newcommand{\bigGiven}{\kern-.15em\bigm\vert\kern-.15em}
\newcommand{\BigGiven}{\kern-.15em\Bigm\vert\kern-.15em}
\newcommand{\biggGiven}{\kern-.15em\biggm\vert\kern-.15em}
\newcommand{\BiggGiven}{\kern-.15em\Biggm\vert\kern-.15em}

\usepackage{xcolor}
\usepackage{graphicx}
\usepackage{tikz}
\usetikzlibrary{positioning}
\usepackage[labelformat=simple]{subcaption}
\usepackage{float} 

\definecolor{myblue}{RGB}{0, 119, 180}

\usepackage[normalem]{ulem}

\renewcommand{\epsilon}{\varepsilon}

\newcommand{\cC}{{\cal C}}

\newcommand{\cE}{{\cal E}}
\newcommand{\cF}{{\cal F}}

\newcommand{\cI}{{\cal I}}

\newcommand{\cN}{{\cal N}}

\newcommand{\cP}{{\cal P}}
\newcommand{\cQ}{{\cal Q}}

\newcommand{\cT}{{\cal T}}

\usepackage[backend = biber,
            giveninits = false,
            maxbibnames = 10,
            doi = false, 
            isbn = false, 
            url = false, 
            eprint = false,
            style = ext-numeric, 
            ]{biblatex}
\DeclareNameAlias{sortname}{given-family} 
\renewbibmacro{in:}{} 
\DeclareDelimFormat{finalnamedelim}{
    \ifnumgreater{\value{liststop}}{2}
    {\unskip,\addspace\&} 
    {\&} 
}

\DeclareFieldFormat[article,inbook,incollection,inproceedings,patent,thesis,unpublished]{titlecase:title}{\MakeSentenceCase*{#1}}

\DeclareFieldFormat[article,inbook,incollection,inproceedings,patent,thesis,unpublished]{title}{``#1\addperiod''}

\DeclareFieldFormat[article,inbook,incollection,inproceedings,patent,thesis,unpublished]{title}{`#1'\addperiod}
\let\Indicator\relax
\DeclareMathOperator{\Indicator}{\mathbb{I}}

\newcommand{\costUnder}{c_{\mathrm{u}}}
\newcommand{\costOver}{c_{\mathrm{o}}}
\renewcommand{\rxi}{\xi}

\usepackage{amsmath}
\newcommand{\liftedProb}{\bm{\mathrm{P}}}
\newcommand{\liftedQrob}{\bm{\mathrm{Q}}}

\newcommand{\repsilon}{\varepsilon}

\newcommand{\pathConstant}{\mu}

\DeclareMathOperator{\logRange}{\texttt{LogRange}}
\DeclareMathOperator{\linRange}{\texttt{LinRange}}

\makeatletter
\def\@maketitle{%
  \newpage
  \begin{center}%
  \let \footnote \thanks
    {\LARGE \@title \par}%
    \vskip 1.5em%
    {\large
      \lineskip .5em%
      \begin{tabular}[t]{c}%
        \@author
      \end{tabular}\par}%
  \end{center}%
  }
\makeatother

\title{\textbf{Nonstationary Distribution Estimation\\via {Wasserstein} Probability Flows}}
\author{\textsl{Edward J.\ Anderson} \\ {\small Imperial College London}\\ {\small \texttt{e.anderson@imperial.ac.uk}}\\ \and \textsl{Dominic S.\ T.\ Keehan} \\ {\small University of Auckland} \\ {\small \texttt{dkee331@aucklanduni.ac.nz}}}

\date{\small 8\textsuperscript{th} of July, 2025}

\begin{document}

\maketitle

\noindent \textbf{Abstract:} We study the problem of estimating a sequence of evolving probability distributions from historical data, where the underlying distribution changes over time in a nonstationary and nonparametric manner. To capture gradual changes, we introduce a model that penalises large deviations between consecutive distributions using the Wasserstein distance.  This leads to a method in which we estimate the underlying series of distributions by maximizing the log-likelihood of the observations with a penalty applied to the sum of the Wasserstein distances between consecutive distributions. We show how this can be reduced to a simple network-flow problem enabling efficient computation. We call this the Wasserstein Probability Flow method. We derive some properties of the optimal solutions and carry out numerical tests in different settings. Our results show that the Wasserstein Probability Flow method is a promising tool for applications in nonstationary stochastic optimization.


\section{Introduction}
In many real-world settings, observed data are generated from an underlying probability distribution that evolves gradually over time. When these changes are smooth rather than abrupt, it remains possible to infer meaningful information about the current distribution from historical observations. In this paper, we address the problem of estimating such an evolving distribution -- particularly its value at the most recent time period -- either for the purposes of forecasting, or as an input to a downstream stochastic optimization problem.

There is a substantial body of literature addressing stochastic optimization problems under nonstationary distributions. Nonstationarity may arise in different ways: one common approach models the underlying distribution as changing a limited number of times. A typical framework allows at most $O(\log(T))$ changes among $T$ historical observations, so that as $T$ gets larger, there are more opportunities to learn the current distribution before another change occurs \parencite{chen2021data}. 

Nonstationary versions of the classical newsvendor problem have been studied extensively. One class of models permits changes in the distribution's mean, constrained by a fixed budget of variation. Variation is measured by taking the maximum over any time-epoch subsequence of the sum of squared distances between consecutive means (akin to a sum of squared peak-to-trough distances) \parencite{keskin2023nonstationary, an2025nonstationary}. Similar formulations are also used for pricing problems where sellers face nonstationary demand \parencite{keskin2017chasing}.

An alternative approach to modelling nonstationarity assumes that the distribution is governed by an underlying state which evolves according to a Markov chain. This gives rise to partially observable Markov decision processes, as studied in the inventory setting of \parencite{treharne2002adaptive}. Nonstationarity in the transition dynamics themselves can also be captured by imposing variation bounds on the transition probabilities, as in \parencite{cheung2023nonstationary}.  
 
The fundamental problem of estimating a nonstationary distribution is often approached parametrically. For example, when the underlying distribution is known to be normal and its mean evolves over time according to a Gaussian process, the Kalman filter provides a natural estimation framework. Indeed, the Kalman filter may also be used as an estimation procedure without strict normality assumptions \parencite{humpherys2012fresh}. But this approach uses a latent state with noisy dynamics and measurements, whereas our interest is in estimating the distribution itself directly from the observed data.

This problem can also be viewed from the perspective of time-series forecasting, where the distribution of the next period's observation is effectively estimated from the forecast mean and variance (a point forecast and associated confidence interval). This is a nonstationary estimation problem where both the mean and variance may change. Many forecasting models incorporate error terms that exhibit heteroscedasticity -- 
this is the basis of the celebrated ARCH and GARCH models, which have proven highly effective in modelling financial time series. But these models differ from ours in that the error term with changing variance relates to the innovation. There are many versions of these models   \parencite{bollerslev1992arch}; typically the time-series value is given by an autoregressive term $y_t=a_0+\sum_{s=1}^{t-1}a_s y_{s}+ \varepsilon_t$
in which $\varepsilon_t$ is drawn from a distribution with zero mean and variance $\sigma_t^2$, and $\sigma_t^2$ depends on the previous values of both $\varepsilon_t$ and $\sigma_t^2$. In an ARCH model the $y_t$ value is the actual value at time $t$, whereas our model combines a random innovation term in the underlying mean value $y_t$ with a measurement error, giving an observation~$\xi_t=y_t+\eta_t$, where the distribution of the measurement error $\eta_t$ can vary over time~as~well.

Two simple and widely used approaches for dealing with evolving data are the \emph{rolling window}, where estimation is based solely on the most recent observations, and its generalisation, \emph{weighted estimation}, where more recent observations are given higher weight \parencite{keskin2023nonstationary}. These approaches can also be adjusted dynamically to reflect the stability of the observations \parencite{huang2023stability}.

When varying weights are applied to observations, a particularly appealing choice is exponential decay, as used in simple exponential smoothing for time-series forecasting. Exponential smoothing is optimal in a variety of settings; see \textcite{chatfield2001new} for a comprehensive review. These include quite general cases. For example, simple exponential smoothing provides the optimal forecast when the observation $\xi_t$ is given by an underlying level $y_t$ plus a noise $\varepsilon_t$, and the $y_t$ are subject to random innovations $\delta_t$ from one period to the next,  with a correlation between~$\varepsilon_t$~and~$\delta_t$, and successive values of both the noise and innovation each independent and identically distributed.  

In standard forecasting applications, exponential smoothing generates a single-point forecast. In contrast, our approach applies exponentially decaying weights to the empirical distribution of observations, producing a new (still discrete) weighted empirical distribution, which serves as an estimate of the evolving underlying distribution. For example, in a nonstationary newsvendor setting we choose an order quantity based on the appropriate quantile of this unequally weighted empirical distribution. Similar ideas are explored in \cite{boudoukh1998best,amrani2017estimation}.
 
Our approach to nonstationary distribution estimation is fully nonparametric. To model changes in the distribution over time, we introduce a measure of temporal variation using the Wasserstein distance.  Specifically, starting from a given distribution, we assume that the likelihood of transitioning to a new distribution is determined by the Wasserstein distance to that distribution, with the probability decreasing as the distance increases. 

In the case of independent samples drawn from a \emph{fixed} distribution, the maximum likelihood estimator is the empirical distribution that assigns equal weight to each sample. This gives the widely used sample average approximation (SAA) method in data-driven stochastic optimization, where each historical observation is treated equally. Thus, SAA can be interpreted as a maximum likelihood procedure under the assumption of distributional stationarity.

The SAA approach is nonparametric and is effective even in circumstances when the true underlying distribution is known to be continuous rather than discrete. One might imagine that additional knowledge about the true distribution (for example Lipschitzian properties of the density) could be used to improve the stochastic optimization procedure, but SAA performs very well even in comparison to much more complex approaches~\parencite{anderson2020can}. 

When the underlying distribution evolves over time, we propose an extension of the maximum-likelihood framework that penalises large changes in the estimated distributions between consecutive time periods. 
As this shift penalty approaches infinity, the method reduces to standard SAA, lending support to its effectiveness even in more general settings where additional distributional structure may be known.

The shift penalty we propose is proportional to the Wasserstein distance between consecutive distributions. Thus, for historical observations $\widehat{\xi}_1,\ldots,\widehat{\xi}_T$, we estimate  distributions~$\mathbb{P}_1,\ldots,\mathbb{P}_T$ by solving
\begin{equation*}
\maximize_{\mathbb{P}_1,\ldots,\mathbb{P}_T} ~~ \sum_{t=1}^{T} \log\bigl(\mathbb{P}_t(\{\widehat{\xi}_t\})\bigr) - \lambda \sum_{t=1}^{T-1} W(\mathbb{P}_t, \mathbb{P}_{t+1}),
\end{equation*}
where $W$ denotes the Wasserstein distance and $\lambda$ is a regularization parameter that controls the trade-off between goodness-of-fit and temporal smoothness.

An alternative formulation would solve the problem with a fixed budget for the sum of the Wasserstein distances between consecutive distributions, and this is closer to some previous work on nonstationarity. But dualising such a budget constraint leads to a penalty in the objective function of the form presented above. Yet another formulation has a fixed limit on the Wasserstein distance between successive distributions, and (in a different context) this is the approach taken in \parencite{keehan2025dontlookangerwasserstein}.

The first-order Wasserstein distance is used to compare consecutive distributions and, because of the triangle inequality, this gives rise to a natural interpretation in terms of flow through a network. This insight forms the basis of our proposed method, which we term the \emph{Wasserstein Probability Flow} (WPF). Unlike other measures of divergence -- such as the Kullback--Leibler divergence, the Wasserstein distance leads to a tractable optimization problem that can be efficiently solved using standard network-flow techniques. Moreover, it also incorporates geometric information contained in the distances between each of the 
observations (which is 
neglected by other divergence-based approaches), and gives the user flexibility to vary the metric to suit a particular application.

We make three contributions in this paper:
\begin{enumerate}[label={(\roman*)}]
\item We introduce the WPF method for estimating probability distributions from nonstationary time series. To the best of our knowledge, this approach is novel and has not appeared in the existing literature.
\item We establish key properties of the WPF solution. In particular, we show that the resulting distribution is uniquely determined under mild conditions and (effectively) supported only on the observed data points. 
We also derive bounds on the probability mass assigned to certain subsets of observations and prove that the probabilities assigned to individual observations are weakly decreasing in their age.  
\item We demonstrate through numerical experiments that the WPF method performs well on both synthetic and real-world data sets.
\end{enumerate}

The remainder of the paper is organised as follows. In Section~\ref{section:model} we present the model underlying the WPF method and show that in each time period the distribution we find at optimality is discrete. Section~\ref{section:network-flow} gives the network-flow formulation. 
In Section~\ref{section:analysis-of-components} we derive some analytical properties of optimal solutions, and then in Section~\ref{section:numerical-results} we report on  numerical tests of the performance of the WPF method. Section~\ref{section:conclusion} concludes with a discussion. All source code and results are available in the repository accompanying this work.%
\footnote{Link: \url{https://github.com/dominickeehan/wasserstein-probability-flows}.}

\section{Model} \label{section:model}
\noindent In this section we give details on our model setup and introduce the WPF method. Let $\Xi$ be a closed subset of $\mathbb{R}^m$. We write $\mathfrak{P}(\Xi)$ for the set of probability distributions on $\Xi$ and $\pmProb_{\xi}$ for the point-mass distribution that assigns probability $1$ to $\xi \in\Xi$. 
All of the distributions considered in this paper are Borel probability measures. For a Polish metric space $(\Xi,d)$, the \emph{Wasserstein distance} between two distributions $\mathbb{P},\mathbb{Q}\in\mathfrak{P}(\Xi)$ is
\begin{equation*}
    W(\mathbb{P},\mathbb{Q}) \defeq \inf_{\gamma \in \Gamma(\mathbb{P},\mathbb{Q})}\int_{\Xi\times\Xi} d(\xi,\zeta) \,\drv\gamma(\xi,\zeta),
\end{equation*}
where $\Gamma(\mathbb{P},\mathbb{Q})$ is the set of all couplings of $\mathbb{P}$ and $\mathbb{Q}$; that is, the set of all probability distributions 
on $\Xi\times\Xi$ with first marginal $\mathbb{P}$ and second marginal $\mathbb{Q}$. We assume throughout the paper that the distance $d$ is lower semicontinuous which implies that the infimum is attained \parencite[Theorem~4.1]{Optimal-Transport:Villani}. 

Suppose that there is a sequence of underlying distributions $\{\mathbb{P}_t \in \mathfrak{P}(\Xi)\}_{t\in[T]}$ generated from some unknown stochastic process. For some $\lambda \geq 0$ we set the likelihood of a shift from $\mathbb{P}_t$~to~$\mathbb{P}_{t+1}$ proportional to
\begin{equation*}
\exp\bigl(-\lambda W(\mathbb{P}_t, \mathbb{P}_{t+1})\bigr).
\end{equation*}
Under the assumption that shifts between periods are independent of the history of previous distributions, the likelihood of realising the sequence $\{\mathbb{P}_t\}_{t\in[T]}$ given $\mathbb{P}_1$ is proportional to
\begin{equation*}
    \prod_{t=1}^{T-1} \exp\bigl(-\lambda W(\mathbb{P}_t, \mathbb{P}_{t+1})\bigr).
\end{equation*}
Our discussion here is informal. We deal with likelihoods and so the relative probabilities are the key quantities. But to do this in a formal sense would require the equivalent of a density on the infinite-dimensional space of $T$-tuples of distributions on $\Xi$, and this is beyond the scope of this paper.

Randomly sampling an observation from each of the distributions $\mathbb{P}_1,\ldots,\mathbb{P}_T$, the likelihood of observing the outcomes $\xi_1,\ldots,\xi_T\in\Xi$ is the product $\prod_{t=1}^{T} {\mathbb{P}_t}(\{\xi_t\})$. Hence, the likelihood of realising the sequence $\{\mathbb{P}_t\}_{t\in[T]}$ given $\{\xi_t\}_{t\in[T]}$ is proportional to
\begin{equation*}
    \prod_{t=1}^{T} {\mathbb{P}_t}(\{\xi_t\}) \cdot \prod_{t=1}^{T-1} \exp\bigl(-\lambda W(\mathbb{P}_t, \mathbb{P}_{t+1})\bigr).
\end{equation*}
Thus, for 
observations $\widehat{\xi}_1,\ldots,\widehat{\xi}_T\in\Xi$, we solve the problem
\begin{alignat}{2}\label{problem:WPF}
    & \kern-0.07em \maximize_{\mathbb{P}_1,\ldots,\mathbb{P}_T\in\mathfrak{P}(\Xi)} && ~~ \sum_{t=1}^{T} \log\bigl(\mathbb{P}_t(\{\widehat{\xi}_t\})\bigr) - \lambda  \sum_{t=1}^{T-1} W(\mathbb{P}_t, \mathbb{P}_{t+1}) \tag{WPF} 
\end{alignat}
to obtain a maximum-likelihood estimate for the sequence of underlying distributions. This has the natural interpretation of simply maximizing log-likelihood with a penalty on changes in the distribution between time periods. Note that we may have $\mathbb{P}_t(\{\widehat{\xi}_t\}) = 0$, so in order for (\ref{problem:WPF}) to be well defined, throughout the paper we set $\log(0) \defeq -\infty$ and $\exp(-\infty) \defeq 0$. 

We also consider a version of the problem with more than one observation per time period. Suppose observation $\widehat{\xi}_t$ is made in period $s(t)\in[S]$, with $s(t)$ increasing in $t$ and at least one observation in each of the $S$ periods. Here $s(1)=1, \ldots, s(T)=S$, so for this version of the problem we simply replace $\mathbb{P}_t$ with $\mathbb{P}_{s(t)}$ in (\ref{problem:WPF}) to give the problem
\begin{alignat}{2}\label{problem:WPF-G}
    & \kern-0.07em \maximize_{\mathbb{P}_1,\ldots,\mathbb{P}_S\in\mathfrak{P}(\Xi)} && ~~ \sum_{t=1}^{T} \log\bigl(\mathbb{P}_{s(t)}(\{\widehat{\xi}_t\})\bigr) - \lambda  \sum_{t=1}^{T-1} W(\mathbb{P}_{s(t)}, \mathbb{P}_{s(t+1)}) \tag{TG-WPF}.
\end{alignat}
We call this the time-grouped version of WPF. Observe that this is the original problem with the additional constraint $\mathbb{P}_{t}=\mathbb{P}_{t+1}$ imposed whenever $s(t)=s(t+1)$.

\begin{remark}[Extreme Values of $\lambda$]
Large values of $\lambda$ correspond to a large penalty on changes in the distribution between periods. If $\lambda = \infty$, then every distribution in the optimal solution to (\ref{problem:WPF}) is the empirical distribution on the observations. On the other hand, if there is no penalty on changes in the distribution between periods, so $\lambda = 0$, then the sequence of point-mass distributions on each observation is optimal.
\end{remark}

Through the definition of the Wasserstein distance, (\ref{problem:WPF}) has minimizations in its objective function. These terms are multiplied by $-\lambda \leq 0$, and thus we include them in the overall maximization. For a feasible solution to (\ref{problem:WPF}) with distributions $\Prob_1,\ldots,\Prob_T$, 
since the infimum in the definition of the Wasserstein distance is attained (as assumed throughout), there exist consecutive minimal-cost transportation plans $\gamma_1,\ldots,\gamma_{T-1}\in\mathfrak{P}(\Xi^2)$ between  $\Prob_1,\ldots,\Prob_T$. Applying the gluing lemma \parencite[page~23]{Optimal-Transport:Villani}, it follows that a distribution $\liftedProb\in\mathfrak{P}(\Xi^T)$ with associated marginals $\Prob_1,\ldots,\Prob_T$ and joint marginals $\gamma_1,\ldots,\gamma_{T-1}$ exists. Let $\Indicator$ denote the event indicator; that is $\Indicator \{\cE\} =1$ if the event $\cE$ occurs and $\Indicator \{\cE\}=0$ if it does not. Expressing $\log \bigl(\Prob_t(\{\widehat{\xi}_t\})\bigr)$ as $\log\bigl(\int_{\Xi^T} \Indicator\{\xi_t = \widehat{\xi}_t\}  \,\drv\liftedProb(\xi_1,\ldots,\xi_T)\bigr)$ and $W(\Prob_t,\Prob_{t+1})$ as $\int_{\Xi^T} d(\xi_t, \xi_{t+1}) \,\drv\liftedProb(\xi_1,\ldots,\xi_T)$ enables us to reformulate (\ref{problem:WPF}) as
\begin{alignat}{2}\label{problem:lifted-discrete-support}
    & \kern-0.07em \maximize_{\liftedProb\in\mathfrak{P}(\Xi^T)} && ~~ \sum_{t=1}^{T} \log \biggl( \int_{\Xi^T} \Indicator\{\xi_t = \widehat{\xi}_t\}  \,\drv\liftedProb(\xi_1,\ldots,\xi_T) \biggr) - \lambda  \sum_{t=1}^{T-1} \int_{\Xi^T} d(\xi_t, \xi_{t+1}) \,\drv\liftedProb(\xi_1,\ldots,\xi_T).
\end{alignat}
Here the distributions $\Prob_t$ in the original formulation are simply the marginals of $\liftedProb$.

\subsection*{Finite-Dimensional Reduction}
Optimizing over arbitrary probability distributions in $\mathfrak{P}(\Xi)$ is an infinite-dimensional problem.  
Our first result relies on the lower semicontinuity of the distance $d$ through the reformulation (\ref{problem:lifted-discrete-support}), and shows that no suboptimality is imposed by instead optimizing over distributions supported only at a finite number of points. Let $\mathfrak{P}_{\nu}(\Xi)\subseteq \mathfrak{P}(\Xi)$ denote the set of discrete probability distributions 
supported on at most $\nu$ points within $\Xi$.
\begin{lemma}\label{lemma:discrete-support}
    For any feasible solution to (\ref{problem:WPF}), there exists another feasible solution consisting of discrete distributions each supported on at most $T+1$ points in \(\Xi\), with the same objective~value. 
\end{lemma}

\begin{proof}
Let $(\Prob_1,\ldots,\Prob_T)$ be a feasible solution to (\ref{problem:WPF}) and let $\liftedProb$ be a feasible solution to (\ref{problem:lifted-discrete-support}) with marginals $\Prob_1,\ldots,\Prob_T$. For this solution $\liftedProb$ the objective function of (\ref{problem:lifted-discrete-support}) features $T$ log terms and we suppose these have values $c_1,\ldots,c_T\in \reals\cup\{-\infty\}$, so $c_t = \log \bigl(\int_{\Xi^T}  \Indicator\{\xi_t=\hat\xi_t\}\,d\liftedProb\bigr)$.  

Replacing each log term in the objective by an equality constraint on its value and exponentiating, we obtain an equivalent linear program over probability measures $\liftedQrob \in\mathfrak{P}(\Xi^T) $:
\begin{alignat}{2}\label{problem:lifted-problem-with-constraint-objective}
    & \kern-0.07em \maximize_{\liftedQrob\in\mathfrak{P}(\Xi^T)} && \quad - \lambda \sum_{t=1}^{T-1} \int_{\Xi^T} d(\xi_t, \xi_{t+1}) \,\drv{\liftedQrob}(\xi_1,\ldots,\xi_T), \\
    & \text{subject to} && \quad \int_{\Xi^T} \Indicator\{\xi_t = \widehat{\xi}_t\}  \,\drv{\liftedQrob}(\xi_1,\ldots,\xi_T) = \exp(c_t), \quad t\in[T]. \notag
\end{alignat}
The problem (\ref{problem:lifted-problem-with-constraint-objective}) is an infinite-dimensional linear program with $T$ equality constraints in the space of Borel probability measures on $\Xi^T$.
By a standard finite-support result for such linear programs (see, e.g., \cite{pinelis-2016-extreme-points-of-moment-sets} or its application in \parencite[Appendix~B]{yue2022linear}), there exists an optimal solution~\(\liftedQrob^\star\) supported on at most \(T+1\) points in~\(\Xi^T\). The marginals of \(\liftedQrob^\star\) then define discrete distributions~\(\Prob_1^\star,\ldots,\Prob_T^\star\) that are feasible for \eqref{problem:WPF}
and achieve the same objective value.
\end{proof}

Our next result makes use of the fact that the distance $d$ satisfies the triangle inequality and shows that no suboptimality is imposed by instead optimizing over distributions supported only on points within the set of observations. The result is established by taking probability mass that is not at an observed point and moving it to an observed point, with the triangle inequality ensuring that overall transport costs do not increase.

\begin{proposition}\label{proposition:discrete-support-on-observations}
For any feasible solution to (\ref{problem:WPF}), there exists a solution consisting of discrete distributions supported only on points within the set of observations with either the same or an improved objective value.
\end{proposition}

\begin{proof}
By Lemma~\ref{lemma:discrete-support}, we may restrict attention to discrete solutions of (\ref{problem:lifted-discrete-support}). Consider such a solution $\liftedProb\in\mathfrak{P}_{T+1}(\Xi^T)$  with  $\liftedProb\bigl(\{(\xi_1,\ldots,\xi_T)\}\bigr)>0$ for some point  $(\xi_1,\ldots,\xi_T)\in\Xi^T$, with $\xi_t \notin \{\widehat{\xi}_1,\ldots,\widehat{\xi}_T\}$ for some $t \in [T]$, where the observations are $\widehat{\xi}_1,\ldots,\widehat{\xi}_T$. This assignment of probability mass does not contribute to the $t$-{th} log term in the objective function of  (\ref{problem:lifted-discrete-support}), so changing $\xi_t$ only impacts the objective through the distance term. We consider two types of change. First suppose $\xi_{t+1} \in \{\widehat{\xi}_1,\ldots,\widehat{\xi}_T\}$. Then the sum of the distance terms in the objective function would not be increased by instead assigning the probability $\liftedProb\bigl(\{(\xi_1,\ldots,\xi_T)\}\bigr)$ to the alternate point $(\xi_1,\ldots,\xi_{t-1},\xi_{t+1},\xi_{t+1},\ldots,\xi_T)\in\Xi^T$, since by the triangle inequality,
\[
  d(\xi_{t-1},\xi_{t+1})
   \le d(\xi_{t-1},\xi_t)+d(\xi_t,\xi_{t+1}).
\]
For the second type of change we suppose that $\xi_{t-1} \in \{\widehat{\xi}_1,\ldots,\widehat{\xi}_T\}$. Then we can assign the probability $\liftedProb\bigl(\{(\xi_1,\ldots,\xi_T)\}\bigr)$ to the alternate point $(\xi_1,\ldots,\xi_{t-1},\xi_{t-1},\xi_{t+1},\ldots,\xi_T)\in\Xi^T$ also without increasing total cost.

Provided that there is at least one element $\xi_s$ of $(\xi_1,\ldots,\xi_T)$ with $\xi_{s} \in \{\widehat{\xi}_1,\ldots,\widehat{\xi}_T\}$, then repeatedly making the first type of change for the largest $t<s$ such that  $\xi_t$ is not within the observations, or making the second type of change for the smallest $t>s$ such that $\xi_t$ is not within the observations, will produce a solution supported only on the observations which is at least as good. If on the other hand $\xi_1, \ldots , \xi_{T} \notin \{\widehat{\xi}_1,\ldots,\widehat{\xi}_T\}$, then this point does not contribute positively to the objective function, and the probability $\liftedProb\bigl(\{(\xi_1,\ldots,\xi_T)\}\bigr)$ can be assigned to $(\widehat{\xi}_1,\ldots,\widehat{\xi}_1)\in\Xi^T$, thereby improving the objective value.

Hence, all probability mass can be concentrated on tuples whose coordinates belong to the observed samples $\{\widehat{\xi}_1,\ldots,\widehat{\xi}_T\}$, giving a discrete optimal solution supported only on the~observations. 
\end{proof}

\section{Network-Flow Reformulation} \label{section:network-flow}
Building on the finite-dimensional reduction of Proposition~\ref{proposition:discrete-support-on-observations}, in this section we show how (\ref{problem:WPF}) can be reformulated as a network-flow problem, and further how it can be reduced in size without altering the set of terminal distributions that can occur in an optimal solution. With Proposition~\ref{proposition:discrete-support-on-observations} in hand, noting that the Wasserstein distance between two discrete distributions can be found through a transportation linear program, we write $p_t(i)$ for the amount of probability mass assigned to observation $i$ by $\mathbb{P}_t$, $\gamma_t(i,j)$ for the amount of probability mass transported between observations $i$ and $j$ during the shift from $\mathbb{P}_t$ to $\mathbb{P}_{t+1}$, and $d(i,j)$ for the distance between observations $i$ and $j$. Thus, to find an optimal solution to (\ref{problem:WPF}), we solve the problem: 
\begin{alignat}{4}\label{problem:network-flow}
    & \kern-0.07em \maximize_{\substack{{p}_1,\ldots,{p}_T\\\gamma_1,\ldots,\gamma_{T-1}}} &&\quad \sum_{t=1}^{T} \log(p_t(t)) - \lambda \sum_{t=1}^{T-1}\sum_{j=1}^{T}\sum_{i=1}^{T} d(i,j)\gamma_t(i,j), \qqquad \\
    & \text{subject to} &&\quad \sum_{i=1}^{T} p_1(i) = 1, \qqquad  \notag\\
    & &&\quad p_{t}(i) = \sum_{j=1}^{T} \gamma_t(i,j), \qqquad i\in[T], \;  t\in[T-1],\notag\\
    & && \quad \sum_{i=1}^{T} \gamma_t(i,j) = p_{t+1}(j), \qqquad j\in[T], \; t\in[T-1],\notag\\
    & &&\quad {p}_t(i) \in \mathbb{R}_+,\vphantom{\sum^T_1} \qqquad  i\in[T], \; t\in[T], \notag\\
    & &&\quad \gamma_t(i,j) \in \mathbb{R}_+, \qqquad  i,j\in[T], \; t\in[T-1]. \notag
\end{alignat}

\noindent This is a type of convex network-flow problem. The probability mass moving through the network is conserved because of the constraints and so the requirement that the probabilities sum to $1$ in the first period is retained for all later periods. It can be seen that this problem has an optimal solution due to the extreme-value theorem, and we therefore conclude that (\ref{problem:WPF}) has an optimal solution as well.

In this network, flow is between an observation $i$ at time $t$ and an observation $j$ at time~$t+1$, i.e., on arcs between $(t,i)$ and $(t+1,j)$. There is conservation of flow at all nodes $(t,i)$ except the sources and sinks at times $t=1$ and $T$. We define a \emph{path} as a sequence of nodes at consecutive time intervals $(t,i_0), (t+1,i_1), \ldots, (t+k,i_k)$. Observe that the flow in the network does not contain any cycles, so from the flow decomposition property \parencite{ahuja1993network}, it can be decomposed into flows along paths (a \emph{path flow}) from nodes at time $t=1$ to time $t=T+1$. We use this property to prove the result below which restricts the types of flow considered at optimality. Since we are interested in estimation applications which use the terminal distribution, a key property is that the restrictions we make do not alter the terminal probabilities 
at optimality.

\begin{proposition}\label{prop:flow-simplification}
For any optimal solution to (\ref{problem:network-flow}), we can construct a new optimal solution in which $p_T$ is unchanged and in which $\gamma_t(i,j)=0$ unless either $i=j$ or $i=t$. 
\end{proposition}

\begin{proof}
We prove this result by construction, starting from an optimal solution and noting that if any $p_T(i)$ is changed then there will be a strict improvement in the objective value. For an optimal solution $(p^{\star}, \gamma^{\star})$, let $\cF$ be a path-flow decomposition. Consider $f\in\cF$. This involves a flow $\delta$ along a path $(1,i_1),  \ldots, (T,i_T)$ and we write $\cT \defeq \{ t : i_t=t \}$ for the set of times when this flow contributes to the log terms. We also write the elements of $\cT$ as an ordered set $\{t_1, \ldots, t_{S} \}$ with $t_1 <  \cdots < t_{S}$.

We now replace the flow $f$ with a new flow $f^{\prime}$ satisfying the statement of the Lemma. For all times $t_i < t < t_{i+1}$ we set $f^{\prime}\bigl((t,t_{i+1}),(t+1,t_{i+1})\bigr)=\delta$ and the other $f^{\prime}$ values to zero for arcs from $t$ to $t+1$. For all $t=t_i$, we set $f^{\prime}\bigl((t,t),(t+1,t_{i+1})\bigr)=\delta$ and the other $f^{\prime}$ values to zero for arcs from $t$ to $t+1$. If $t_1 >1$, then for all $t<t_1$ we set $f^{\prime}\bigl((t,t_{1}),(t+1,t_{1})\bigr)=\delta$ and the other $f^{\prime}$ values to zero for arcs from $t$ to $t+1$. If $t_{S} < T$, then for all $t_{S} \leq t < T$ we set $f^{\prime}((t,t_{S}),(t+1,t_{S}))=\delta$ and the other $f^{\prime}$ values to zero for arcs from $t$ to $t+1$. This change is illustrated in Figure \ref{figure:proof-illustration}.
\begin{figure}[H] 
\centering
\centerline{\includegraphics[trim={0 2.5cm 0 2.5cm},clip,width=19.8cm] {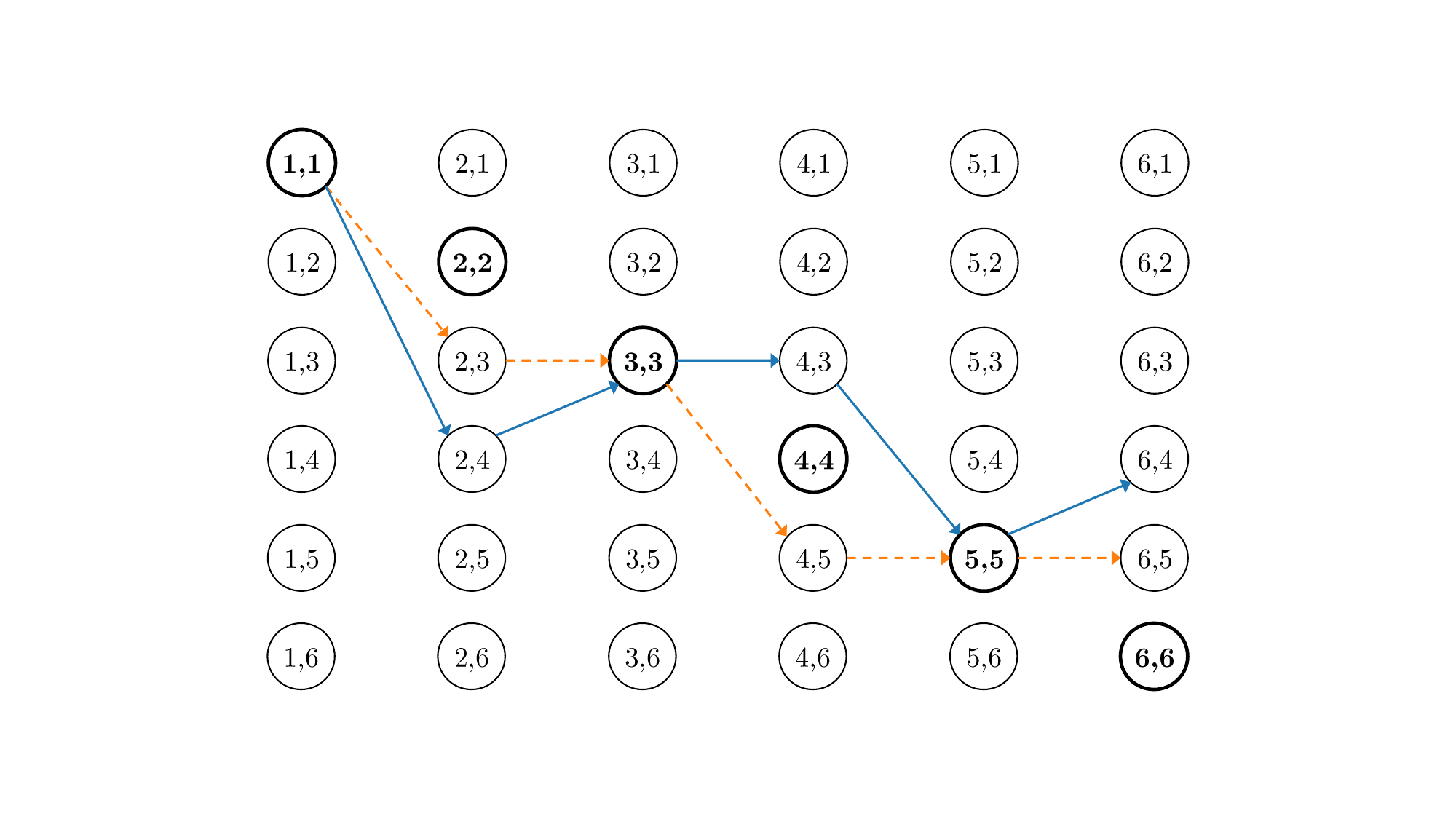}}
\caption{\textbf{Construction Used in the Proof of Proposition~\ref{prop:flow-simplification}.} Solid arrows show the initial flow $f$ and dashed arrows show the newly constructed flow $f^{\prime}$. Bold nodes on the diagonal contribute to the log terms in the objective, so $\cT=\{1,3,5\}$ here.}
\label{figure:proof-illustration}
\end{figure}

\noindent Note that $f^{\prime}$ is a path flow that matches the flow into the nodes where $f$ contributes to the log terms in the objective. Using the triangle inequality, we can also see that the total distance along $f^{\prime}$ is no greater than the distance along $f$, since wherever possible $f^{\prime}$ moves directly to the next node, rather than possibly moving via other nodes. Moreover, there is zero distance along the sections of the path before $t_1$ or after $t_{S}$.

Repeating this construction for each $f\in\cF$ and combining the resulting path flows yields a feasible solution with a no-worse objective value, and in which the only non-negative flows~$\gamma_t(i,j)>0$ occur when $i=j$ or $i=t$. If $p_T^{\star}$ were altered through any of the changes from $f$~to~$f^{\prime}$, then we must have some $t_{S}<T$ and $i_T \neq t_{S}$, in which case there will be nonzero distances in the last section of the path replaced with zero distances. This gives an improvement from the change and therefore contradicts the optimality of $(p^{\star}, \gamma^{\star})$. We conclude that if the original flow is optimal, then so is the new flow.
\end{proof}

\noindent This result enables us to rewrite the problem (\ref{problem:network-flow}) with a reduced set of flow variables from one observation to a later observation. We consider source and sink nodes $0$ and $T+1$, and define a variable $x(i,j)$ for the probability flow between different nodes $i$ and $j$. We set 
\begin{alignat*}{2}
    & x(0,j) = p_1(j), \qquad j\in[T],\\
    &x(i,j) = \gamma_i(i,j), \qquad i\in[j-1], \; j\in[T],\\
    &x(i,T+1) = p_T(i), \qquad i\in[T],
\end{alignat*}
and all other $x(i,j)$ variables to $0$,  applying Proposition~\ref{prop:flow-simplification}. As a consequence,
\begin{equation*}
    \sum_{i=0}^T x(i,j) = p_j(j), \quad j\in[J].
\end{equation*}
Thus, to find the terminal distributions in optimal solutions to (\ref{problem:WPF}), we can solve the problem:
\begin{alignat}{4}\label{problem:reduced-network-flow}
    & \kern-0.07em \maximize_{x} &&\quad \sum_{j=1}^{T} \log\biggl(\ttspace\sum_{i=0}^T x(i,j)\biggr) - \lambda \sum_{j=1}^{T}\sum_{i=1}^{T} d(i,j)x(i,j), \\
    & \text{subject to} &&\quad \sum_{j=1}^{T} x(0,j) = 1, \notag\\
    & &&\quad \sum_{i=0}^T x(i,j) = \sum_{k=1}^{T+1} x(j,k), \qqquad j\in[T],  \notag\\
    & &&\quad x(i,j) = 0\vphantom{\sum^T_1}, \qqquad j\in[i], \; i\in[T], \notag \\
     & &&\quad x(0,T+1) = 0, \vphantom{\sum_1}\qqquad \notag\\
    & &&\quad x(i,j) \in \mathbb{R}_+, \qqquad i\in 0\cup[T], \; j\in[T+1]. \notag
\end{alignat}
\noindent Ignoring variables which are fixed to $0$, the reduced problem (\ref{problem:reduced-network-flow}) has $O(T^2)$ variables and $O(T)$ constraints, whereas the original problem (\ref{problem:network-flow}) has $O(T^3)$ variables and $O(T^2)$ constraints.

We may reformulate (\ref{problem:reduced-network-flow}) by replacing each $\log(p_t)$ term in the objective with a variable $z_t\in\reals$ and imposing the constraint $z_t \leq \log(p_t)$. These amount to exponential-cone constraints, and the result is an exponential-cone program that can be readily addressed with standard solvers.

To solve the time-grouped version of the problem we simply add additional constraints to ensure that there is no flow between observations within the same time epoch.
\begin{proposition}\label{lemma:time-grouped}
The problem~(\ref{problem:reduced-network-flow}) with the additional constraints 
\begin{equation*}
    x(i,j)=0\quad\text{if}\quad s(i)=s(j),\quad i,j\in[T],
\end{equation*}
has the same optimal value as (\ref{problem:WPF-G}). 
\end{proposition}

\begin{proof}
For the time-grouped problem the equivalent of the reformulation (\ref{problem:lifted-problem-with-constraint-objective}) is
\begin{alignat}{2}\label{problem:time-grouped-with-constraint-objective}
    & \kern-0.07em \maximize_{\liftedProb\in\mathfrak{P}(\Xi^S)} && \quad - \lambda \sum_{s=1}^{S-1} \int_{\Xi^S} d(\xi_s, \xi_{s+1}) \,\drv\liftedProb(\xi_1,\ldots,\xi_S), \\
    & \text{subject to} && \quad \int_{\Xi^S} \Indicator\{\xi_{s(t)} = \widehat{\xi}_t\}  \,\drv\liftedProb(\xi_1,\ldots,\xi_S) = \exp(c_t), \quad t\in[T] \notag
\end{alignat}
and we can again conclude that imposing an additional constraint $\liftedProb\in\mathfrak{P}_{T+1}(\Xi^S)$ results in no change to the optimal objective value. Now consider the problem (\ref{problem:WPF-G}) and suppose that $(\xi_1,\ldots,\xi_S)\in\Xi^S$ is a feasible solution satisfying $\xi_s \notin \{\widehat{\xi}_1,\ldots,\widehat{\xi}_T\}$ for some $s \in [S]$. Then the arguments in the proof of Proposition \ref{proposition:discrete-support-on-observations} apply and we establish that we can restrict solutions of (\ref{problem:WPF-G}) to the set of observations. 

For the network-flow reformulation of (\ref{problem:WPF-G}), we obtain (\ref{problem:network-flow}) with the additional constraints that for all $t \in [T-1]$ with $s(t)=s(t+1)$, then $p_t(i)=\gamma_t(i,i)=p_{t+1}(i)$  for $i \in [T-1]$ and $\gamma_t(i,j)=0$ if $j \neq i$. These constraints ensure $\mathbb{P}_{t}=\mathbb{P}_{t+1}$. For this problem we need a slightly different construction to that in Proposition~\ref{prop:flow-simplification}. There is a need to restrict the movement of probability to the first opportunity after the set of observations made at the same time epoch. Thus we have $\gamma_t(i,j)=0$ unless either $i=j$; or $s(t)=i$ and $s(t) \ne s(t+1)$. With these changes and the corresponding transformation of flows $f$ to $f^{\prime}$ the proof of Proposition \ref{prop:flow-simplification} goes through unchanged. The final step to the formulation of (\ref{problem:reduced-network-flow}) with the additional constraint~$x(i,j)=0$~if~$s(i)=s(j)$ follows by setting $x(i,j)=\gamma_t(i,j)$ for all $t$ with $s(t)=i$ and $s(t) \neq s(t+1)$.  
\end{proof}

The following result ensures that (\ref{problem:reduced-network-flow}) has unique  $x(i,T+1)$, $i\in[T]$, values at optimality, thus implying that the optimal terminal distribution for (\ref{problem:network-flow}) is unique, with the same implication for (\ref{problem:WPF}) when restricting to discrete distributions supported only on the observations. We say that a set of scalars $\{c_1,  \ldots,c_N\}$ 
has \emph{unique subset sums} if there are no two subsets $\cI,\cI^{\prime} \subseteq [N]$ with $\cI \neq \cI^{\prime}$ and $\sum_{i \in \cI}c_i = \sum_{i \in \cI^{\prime}}c_i$. 

\begin{theorem} \label{theorem:unique-solution}
    If each set of the nonzero distances along any acyclic path between two different observations (i.e., a set 
    $\{d(i_1,i_2),d(i_2,i_3), \ldots, d(i_{k-1},i_k) \}$ with $i_1,i_2, \ldots, i_k$ all differing) has unique subset sums, then every optimal solution to (\ref{problem:network-flow}) has the same terminal probability distribution, 
    $p_T$; that is, the optimal terminal distribution is unique.
    \label{proposition:unique-P_T}
\end{theorem}

\begin{proof}
We prove this for (\ref{problem:reduced-network-flow}) which establishes the result for (\ref{problem:network-flow}). The problem (\ref{problem:reduced-network-flow}) is a maximization with a concave objective function and linear constraints, so the optimal solutions form a convex set. Consider two optimal solutions $x$ and $x^{\prime}$ that differ on the flows to $T+1$; obtaining a contradiction will be enough to establish the result. Moving along the line of convex combinations of $x$ and $x^{\prime}$, strict concavity of $\log $ implies that $\sum_{i=0}^{T}x(i,j) =\sum_{i=0}^{T}x^{\prime}(i,j)$ for each $j\in[T]$. Thus the objective values for $x$ and $x^{\prime}$ match for the log terms in the objective. Since these are both optimal solutions and have the same objective value, we deduce that
\begin{equation*}
\sum_{j=1}^{T}\sum_{i=1}^{T} d(i,j)x(i,j) = \sum_{j=1}^{T}\sum_{i=1}^{T} d(i,j)x^{\prime}(i,j).
\end{equation*}
Now consider $x-x^{\prime}$: this is a nonzero flow through the network with positive and negative elements and zero net flow in and out of all nodes (including $0$ and $T+1$). It can hence be decomposed into cycles. By assumption, $x-x^{\prime}$ includes a nonzero flow to $T+1$. Among these cycles including $T+1$ we choose $x^{\prime\prime}$ and write $i_{1}$ and $i_{2}$ for the two nodes adjacent to $T+1$.

Note that for any pair $(i,j)$, we have $x^{\prime\prime}(i,j) \neq 0$ only if $x(i,j)- x^{\prime}(i,j)$ is nonzero, and hence at least one of $x(i,j)$ or $x^{\prime}(i,j)$ is nonzero. Now consider solutions $(x+x^{\prime})/2 \pm \delta x^{\prime\prime}$. Because $x^{\prime\prime}$ is a cycle these two solutions satisfy a conservation of flow property at every node other than $0$ and $T+1$, and their components are all non-negative for small enough $\delta$. Thus they are both feasible solutions.

Since $x$ and $x^{\prime}$ are both optimal, we arrive at a contradiction unless 
\begin{equation}\label{equation:cycle-distance-sum-is-zero}
\sum_{j=1}^{T}\sum_{i=1}^{T} d(i,j)x^{\prime\prime}(i,j) = 0.
\end{equation}
(We exclude $x^{\prime\prime}(i_{1},T+1)$ and $x^{\prime\prime}(T+1,i_{2})$ from this sum as the multiplying distance is $0$.) With $x^{\prime\prime}$ being a cycle, each nonzero term has the same magnitude. In view of (\ref{equation:cycle-distance-sum-is-zero}), splitting $x^{\prime\prime}$ into positive and negative elements then contradicts the unique subset sum property for the path from $i_{1}$ to $i_{2}$ that does not visit $T+1$. Hence, our supposition that $x \neq x^{\prime}$ is wrong and the result is established.
\end{proof}

The unique subset sum condition of Theorem~\ref{theorem:unique-solution} is unrestrictive when $\Xi$ is not discrete. In this case the condition will always hold for a small perturbation of the observations. 

In the case where the set of all nonzero distances has the unique subset sum property (which is a stronger condition than that of Theorem \ref{theorem:unique-solution}) we have the stronger result that there is a unique optimal solution to (\ref{problem:reduced-network-flow}). However, this condition tends to fail when $\Xi$ is one dimensional, due to the fact that transporting mass from observation  $\widehat{\xi}_1$ to observation $\widehat{\xi}_2$ has the same cost as going via $\widehat{\xi}_3$ if $\widehat{\xi}_3$ lies between $\widehat{\xi}_1$ and $\widehat{\xi}_2$.

\begin{example}\label{example:small-network}
It will be helpful to consider a small one-dimensional example. Suppose that we have the $6$ observations $6.13$, $7.85$, $6.47$, $4.91$, $5.54$, and $7.13$ (in sequence), and we use the standard Euclidean norm to compute distances. Solving (\ref{problem:reduced-network-flow}) with $\lambda=4$ yields the optimal objective value $-8.7052$ and a set of flows shown in Figure~\ref{figure:example-network-solution}. The terminal distribution assigns probability $0.275$ to observation~2 ($7.85$), probability $0.021$ to observation~3 ($6.47$), probability $0.325$ to observation~5 ($5.54$), and  probability $0.379$ to observation~6 ($7.13$). 
\begin{figure}[H]
\centering
\centerline{\includegraphics[trim={0 1.5cm 0 1.75cm},clip,width=13.75cm]{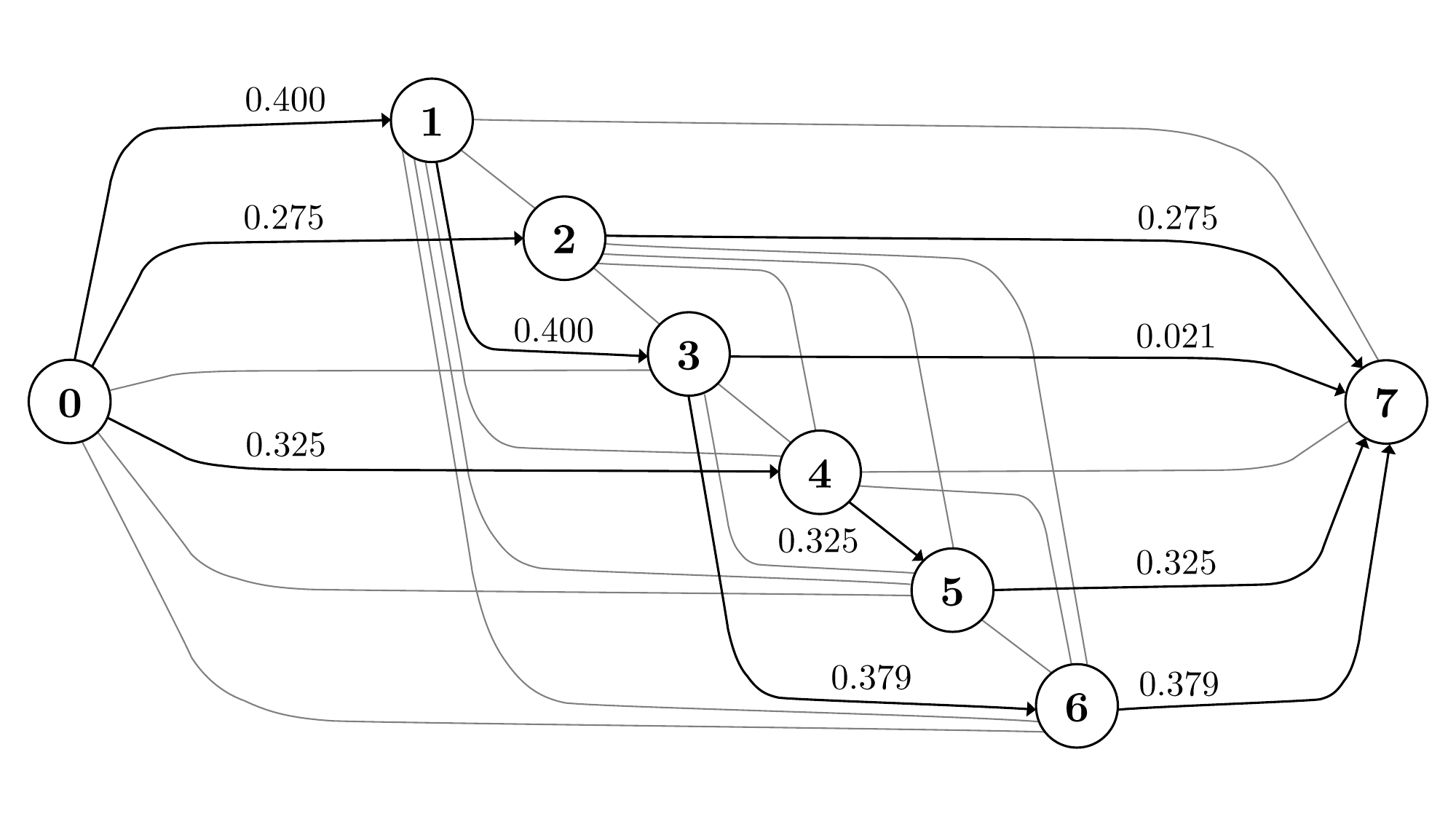}}
\caption{\textbf{Optimal Solution for Example~\ref{example:small-network}.} Arcs that have no flow are shown in grey.}
\label{figure:example-network-solution}
\end{figure}

\noindent Note that reversing time yields the same optimal flow pattern with the directions reversed. Thus, we can also see from the figure the optimal solution to (\ref{problem:reduced-network-flow}) if the observations were instead collected in reverse order. The terminal distribution then assigns probabilities $0.400$, $0.275$ and $0.325$ to observations $6.13$, $7.85$  and $4.91$ respectively. This demonstrates that the order that the observations are collected influences the optimal solution for the WPF method. \hfill \qedsymbol
\end{example}

\section{Analysis of Components}\label{section:analysis-of-components}
The solution of the network-flow problem (\ref{problem:reduced-network-flow}) provides a natural grouping of 
observations into subsets based on the way that probability mass is transported through the network. By studying such groupings, in this section we provide structural results that characterise how the probabilities assigned to each observation are affected by the value of the observations and the periods in which they are observed.

A feasible solution $x$ to (\ref{problem:reduced-network-flow}) defines a graph on the nodes $\{1, \ldots, T\}$ with arcs $(i,j)$ where either $x(i,j) >0$ or $x(j,i)>0$. We say that a \emph{component} $\mathcal{C} \subseteq [T]$ in a feasible solution $x$ is a component of this graph (i.e., a maximal subgraph in which a path between every two~nodes~exists). 

In the example of Figure~\ref{figure:example-network-solution} there are three components: the first component consists of observations~1, 3, and 6 ($6.13$, $6.47$, and $7.13$), the second component is just observation~2 ($7.85$) as a singleton, and the third component consists of observations~4 and 5 ($4.91$ and $5.54$). Within the first component the terminal distribution assigns no probability mass to observation $1$, and within the last component it assigns no probability mass to observation $4$.

\subsection{Bounds on Component Mass} 
It will be helpful to work with a path-based description of the network-flow version of \eqref{problem:WPF}. We  consider all possible paths  $\cP_1,\ldots,\cP_L\subseteq[T]$ between the source and sink nodes, and write $(j,k) \in \cP_\ell$ if $j$ and $k$ appear as successive elements of $\cP_\ell$ (each $\cP_\ell$ is an ordered set). We do not include the source $\{0\}$ or the sink  $\{T+1\}$ in such paths. In this way a solution $x$ to (\ref{problem:reduced-network-flow}) is associated with a set of flows $f_1,\ldots,f_L\in[0,1]$ on these paths with 
$$x(j,k)=\sum_{i : (j,k) \in \cP_i} f_i.
$$  
By identifying the first element in each path we can also write $x(0,j)$ as a sum; that of $f_i$ over paths with $\{j\}$ as the first element. A similar process gives $x(k,T+1)$.
We write $D(\cP_{\ell})$ for the total distance along path $\cP_{\ell}$, so 
$$
D(\cP_{\ell}) =\sum_{(j,k) \in \cP_{\ell}}d(j,k),
$$ 
and we write $\cN (j) \subseteq [L]$ for the set of paths that include node $j$. Thus, another way to state problem (\ref{problem:reduced-network-flow}) is
\begin{equation}\label{WPF-path}
\maximize_{f_1,\ldots,f_L\in\reals_+} \quad \sum_{t=1}^T \log \biggl({\sum}_{i \in \cN(t)} f_i \biggr) -\lambda \sum_{\ell=1}^L f_{\ell} D(\cP_{\ell}), \quad \subjectTo \quad \sum_{\ell=1}^L f_{\ell}=1. 
\end{equation}

\begin{lemma}\label{lemma:path-derivative}
A feasible solution $x$ to (\ref{problem:reduced-network-flow}), with flow decomposition $f_1,\ldots,f_L$ over the possible paths from source to sink $\cP_1,\ldots,\cP_L$, is optimal if and only if there exists a constant $\pathConstant$, such that, for every path $\cP_{\ell} =\{\ell_1, \ldots, \ell_J\}$, the bound
\begin{equation} \label{equation:path-constant}
 \sum_{j=1}^{J}\biggl({\sum}_{i \in \cN(\ell_j)} f_i\biggr)^{-1}-\lambda D(\cP_{\ell}) \le \pathConstant
\end{equation}
holds, with equality if $f_{\ell} >0$.
\end{lemma}

\begin{proof}
The problem (\ref{WPF-path}) is a convex optimization problem.
Dualising the $\sum_{\ell=1}^L f_{\ell} =1$ constraint, there is a constant 
$\pathConstant$ (which is the associated Lagrange multiplier) such that the optimal solution occurs when
\begin{equation*}
\sum_{t=1}^T \log \biggl({\sum}_{i \in \cN(t)} f_i\biggr) -\lambda  \sum_{\ell=1}^L f_{\ell} D(\cP_{\ell})+\pathConstant \biggl(1-\sum_{\ell=1}^L f_{\ell}\biggr)
\end{equation*}
is maximized over $f_1,\ldots,f_L \ge 0$. Taking derivatives, when $f_{\ell} > 0$ 
\begin{equation*}
 \sum_{j=1}^{J}\biggl({\sum}_{i \in \cN(\ell_j)} f_i\biggr)^{-1}-\lambda D(\cP_{\ell}) - \pathConstant =0,
\end{equation*}
and otherwise when $f_{\ell} = 0$, then the derivative is negative, as required. Sufficiency follows from the standard Karush--Kuhn--Tucker argument.
\end{proof}

For a component $\mathcal{C}$ in a solution to (\ref{problem:reduced-network-flow}) we write the total probability mass for the component~as
$$
M(\mathcal{C})=\sum_{i \in \mathcal{C}}x(0,i).
$$
Due to conservation of flow, this is equal to the sum of terminal probabilities on nodes in this component,  i.e., 
$$M(\mathcal{C})=\sum_{i \in \mathcal{C}}x(i,T+1)=\sum_{i \in \mathcal{C}}\mathbb{P}_T (\{\widehat{\xi}_i\}).$$

Our next result bounds the total probability mass in a component in terms of the number of observations that it includes.

\begin{proposition}\label{proposition:component-mass-upper-bound}
For each component $\cC$ in an optimal solution $x$ to (\ref{problem:reduced-network-flow}), the bounds
\begin{equation*}
   \lvert\tspace\cC\tspace\rvert\Bigl(\pathConstant + \lambda D_{\max}(\cC)\Bigr)^{-1} \leq M(\cC) \leq \lvert\tspace\cC\tspace\rvert\tspace\pathConstant^{-1}
\end{equation*}
hold, where $\pathConstant$ is defined as in (\ref{equation:path-constant}) and $D_{\max}(\cC)\defeq\max_{\ell\in[L]}\{D(\cP_\ell):\cP_\ell \subseteq \cC\}$.
\end{proposition}

\begin{proof}

We proceed by taking an optimal solution $x$ to the network-flow problem (\ref{problem:reduced-network-flow})  with a flow decomposition $f_1,\ldots,f_L\in[0,1]$ on the paths $\cP_1,\ldots,\cP_L\subseteq[T]$ between the source and sink. By the definition of  components, it is clear that no path flows include elements from more than~one~component.

For a given component $\cC$ we consider paths $\cQ_1,\ldots,\cQ_K$ in 
$\cC$ with flows $f_1,\ldots,f_K >0$. The total probability mass in this component is simply the sum of the flows: $ M(\cC) = \sum_{k=1}^K f_k$,
and thus we rewrite the result of Lemma \ref{lemma:path-derivative} as 
\begin{equation*}
\sum_{j \in \cQ_k}\biggl({\sum}_{i : j \in \cQ_i} f_i\biggr)^{-1}-\lambda D(\cQ_k) = \pathConstant,
\end{equation*}
for each path $\cQ_k\in \cC$. Hence, 
\begin{equation} \label{eqn:boundA}
 \sum_{k=1}^K \sum_{j \in \cQ_k}f_k\biggl({\sum}_{i : j \in \cQ_i} f_i\biggr)^{-1} = \sum_{k=1}^K f_k \Bigl(\pathConstant+\lambda D(\cQ_k)\Bigr).
\end{equation}
Now, the left-hand side of (\ref{eqn:boundA}) is simply the number of observations in the component, as
\begin{equation*}
\sum_{k=1}^K \sum_{j \in \cQ_k}f_k\biggl({\sum}_{i : j \in \cQ_i} f_i\biggr)^{-1} = \sum_{j \in \cC} \sum_{k : j \in \cQ_k}\biggl({\sum}_{i : j \in \cQ_i} f_i\biggr)^{-1} f_k
 = \sum_{j \in \cC} \biggl({\sum}_{i : j \in \cQ_i} f_i\biggr)^{-1} \sum_{k : j \in \cQ_k}f_k = \lvert\tspace\cC\tspace\rvert.
\end{equation*}
Then, since $\pathConstant \le \pathConstant+\lambda D(\cQ_k) \le \pathConstant + \lambda D_{\max}(\cC)$, the result follows from (\ref{eqn:boundA}).
\end{proof}

Suppose that in an optimal solution $\{i\}$ is a component with only one element (a \emph{singleton component}), which therefore consists of a single path $\cP_\ell = \{i\}$. The only flow through $i$ is from $0$ to $i$ to $T+1$. Since $D(\cP_\ell)=0$, Proposition~\ref{proposition:component-mass-upper-bound} implies $ \Prob_T(\{\widehat{\xi}_i\}) = 1/\pathConstant$. (Equality holds since $D_{\max}(\cC)=0$.)  Proposition~\ref{proposition:component-mass-upper-bound} further establishes that the total probability mass on any component $\cC$ is at most $\lvert\tspace\cC\tspace\rvert \tspace \Prob_T(\{\widehat{\xi}_i\})$, i.e.,  $\lvert\tspace\cC\tspace\rvert$
times as large as that for the singleton component. We typically find that the observations which end up as singleton components are far away from any of the other observations and thus could be described as outliers. For instance, in Example~\ref{example:small-network} it is the observation $7.85$ that ends up as a singleton component.  On the other hand, observations which are part of the same component are typically close together, and as a consequence of the previous reasoning, end up with a lower total probability mass than if they had been outliers. In effect, the WPF method gives greater mass to outliers.

This relative upweighting of outliers is an unusual feature of the WPF method. The intuition is that encountering a set of observations which are close together suggests a single point of mass in the underlying distribution has shifted over time, whereas there is less evidence to suggest that the point of mass responsible for the outlying observation has shifted.

\subsection{Do Older Observations Have Smaller Assigned Probabilities?}
We expect that observations further in the past will generally have smaller probability in the final distribution. This is true only in a weak sense: we show that if the sequence of observations is changed so that a particular observation is made at an earlier point in time, then the probability assigned to that observation in the terminal WPF estimate will either stay the same or~be~reduced, provided that none of the observations which as a result are shifted one time period later are in the same component.

Throughout this subsection we assume the conditions of Theorem~\ref{theorem:unique-solution} ensuring the optimal solution is always unique. We will need some additional notation. Consider an unordered set of observations $\widehat{\Xi} = \{\widehat{\xi}_1, \ldots \widehat{\xi}_T \}$  and let $\sigma: \{1,\ldots , T \}  \rightarrow \widehat{\Xi}$ be a  sequence giving the order in which these occur. Given a particular sequence $\sigma$ we are interested in the optimal terminal distribution $\mathbb{P}_T^{\star}$ of probabilities attached to each of these observations under WPF; we write the optimal terminal probability assigned to observation $\widehat{\xi}_{t}$ under sequence $\sigma$ as~$\Prob^{\star}_T(\{\widehat{\xi}_t\};\sigma)$.

Suppose that we are given a sequence $\sigma$ for the observations and a pair of times $s$ and $s^{\prime}$, with $s^{\prime}<s$, 
we define a new sequence of observations $\sigma^{\prime}$ that moves the observation at time $s$ to the earlier time $s^{\prime}$, at the same time as shifting all the observations made at intermediate times one period later; so $\sigma^{\prime}(s^{\prime})=\sigma(s)$ and $\sigma^{\prime}(t)=\sigma(t-1)$, for $t=s^{\prime}+1, \ldots ,s$, with the other values of $\sigma^{\prime}$ the same as for~$\sigma$. The result below shows that the probability assigned to the observation~$\sigma(s)$ either decreases or stays the same if it is moved earlier in the sequence, provided that none of the observations between period $s^{\prime}$ and period $s-1$ are in the same component as the observation at $s$ in the WPF solution.

\begin{theorem} \label{move-backwards}
If $\sigma(s)=\widehat{\xi}_s$, and the observations  $\sigma(s^{\prime}), \sigma(s^{\prime}+1) \ldots, \sigma(s-1)$ are in different components to $\widehat{\xi}_s$ in the optimal WPF solution for the sequence of observations $\sigma$, then   $\Prob^{\star}_T(\{\widehat{\xi}_s\};\sigma) \ge \Prob^{\star}_T(\{\widehat{\xi}_s\};\sigma^{\prime})$.
\end{theorem}

\begin{proof}
We will reformulate the problem (\ref{problem:reduced-network-flow}) by identifying the variables~$x(i,j)$ with the observations themselves rather than with the order of observations, so we work with variables~$x(\widehat{\xi}_i,\widehat{\xi}_j)$. Henceforth we label the observations using the unswapped ordering $\sigma$, so that $\widehat{\xi}_i = \sigma(i)$ for $i=s^{\prime},\ldots,s$. We will also need an extended set of observations which is $\widehat{\Xi}$ together with additional dummy observations $\widehat{\xi}_0$ and $\widehat{\xi}_{T+1}$ for the source and sink nodes. (Note that these are unaffected by the switch in ordering from $\sigma$ to $\sigma^{\prime}$, so $\hat{\xi}_0 = \sigma(0) = \sigma^{\prime}(0)$ and $\hat{\xi}_{T+1} = \sigma(T+1) = \sigma^{\prime}(T+1)$.) The problem (\ref{problem:reduced-network-flow}) becomes:
\begin{alignat}{4}\label{problem:sequence-network-problem}
    & \kern-0.07em \maximize_{x} &&\quad \pi(x)\defeq
    \sum_{j=1}^{T} \log\biggl(\ttspace\sum_{i=0}^T x(\widehat{\xi}_i,\widehat{\xi}_j)\biggr) - \lambda \sum_{j=1}^{T}\sum_{i=1}^{T} d(\widehat{\xi}_i,\widehat{\xi}_j)x(\widehat{\xi}_i,\widehat{\xi}_j), \\
    & \text{subject to} &&\quad \sum_{j=1}^{T} x(\widehat{\xi}_0,\widehat{\xi}_j) = 1, \notag\\
    & &&\quad \sum_{i=0}^T x(\widehat{\xi}_i,\widehat{\xi}_j) = \sum_{k=1}^{T+1} x(\widehat{\xi}_j,\widehat{\xi}_k), \qquad j\in[T],  \notag\\
    & &&\quad \vphantom{\sum^T_1} x(\widehat{\xi}_i,\widehat{\xi}_j) = 0~~\text{if}~~\sigma^{-1}(\widehat{\xi}_j) < \sigma^{-1}(\widehat{\xi}_i), \qquad i\in 0\cup[T], \; j\in [T+1], \notag \\
     & &&\quad x(\widehat{\xi}_0,\widehat{\xi}_{T+1}) = 0, \vphantom{\sum_1}\qquad \notag\\
    & &&\quad x(\widehat{\xi}_i,\widehat{\xi}_j) \in \mathbb{R}_+, \qquad i\in 0\cup[T], \; j\in [T+1]. \notag
\end{alignat}
The switch
in the order of the observations from $\sigma$ to $\sigma^{\prime}$ leaves the problem \eqref{problem:sequence-network-problem} unchanged except for the constraints on which variables $x(\widehat{\xi}_i,\widehat{\xi}_j)$ are equal to zero: the constraints $x(\widehat{\xi}_s,\widehat{\xi}_k)=0$, $k=s^{\prime},\ldots,s-1$, being dropped and the new constraints $x(\widehat{\xi}_k,\widehat{\xi}_s)=0$, $k=s^{\prime},\ldots,s-1$, being added. Towards obtaining a contradiction, we proceed by deriving a chain of inequalities relating optimal solutions to the unswapped and swapped versions of \eqref{problem:sequence-network-problem}.

Let $x^\star$ be an optimal solution for the original ordering $\sigma$ and $x^{\prime\star}$ an optimal solution for the swapped ordering $\sigma^{\prime}$. By the hypothesis that in the optimal solution for $\sigma$ the observations~$\widehat\xi_{s^{\prime}},\widehat\xi_{s^{\prime}+1},\dots,\widehat\xi_{s-1}$ lie in different components to $\widehat\xi_{s}$, we have that
\[
x^\star(\widehat\xi_k,\widehat\xi_s)=0,\qquad k=s^{\prime},\dots,s-1.
\]
Hence, $x^\star$ satisfies the feasibility constraints of the problem with the swapped ordering~$\sigma^{\prime}$. 

Define $v:=x^{\prime\star}-x^\star$. Note that as a difference of two feasible flows, $v$ satisfies conservation of flow at every node, and thus $v$ is a circulation that may be decomposed into directed cycles~\parencite{ahuja1993network}. (Negative flows are treated as positive flows on the reverse arcs.) 
Let $v_A$ be the sum of flow cycles that include the arc~$(\widehat\xi_s,\widehat\xi_{T+1})$. By construction $v_A$ has conservation of flow at every node and only has outflow from $\widehat\xi_s$ to $\widehat\xi_{T+1}$. Moreover, on all arcs not involving $\widehat\xi_s$, in each coordinate $x^\star+v_A$ lies in the interval between  $x^\star$ and $x^{\prime\star}$, and thus $x^\star+v_A$ is feasible for the problem with the unswapped ordering $\sigma$.

Further, let $v_B \defeq v-v_A$ be the sum of the remaining flow cycles. Since, for each $k=s^{\prime}, \ldots, s-1$, we have that $x^{\prime\star}(\widehat{\xi}_k,\widehat{\xi}_s) = 0$ due to the constraints, and that $x^{\star}(\widehat{\xi}_k,\widehat{\xi}_s) = 0$ due to the hypothesis, the difference 
$v(\widehat{\xi}_k,\widehat{\xi}_s)$ is also equal to $0$. It follows that $v_B(\widehat{\xi}_k,\widehat{\xi}_s)=0$ for $k=s^{\prime}, \ldots, s-1$. For arcs not involving $\widehat\xi_s$, in each coordinate $x^\star+v_B$ lies in the interval between  $x^\star$ and $x^{\prime\star}$, and so $x^{\star}+v_B$ is feasible for the problem with the swapped ordering~$\sigma^{\prime}$. From optimality for the respective problems we therefore have
\begin{equation}\label{equation:from-optimality}
\pi(x^\star)\ge \pi(x^\star+v_A)
\qquad\text{and}\qquad
\pi(x^{\prime\star})\ge \pi(x^\star+v_B).
\end{equation}

For the objective function $\pi$ in \eqref{problem:sequence-network-problem}, consider the value $\pi(x^\star)$ and the perturbed terms $\pi(x^\star+v_A)$, $\pi(x^\star+v_B)$, and $\pi(x^\star+v_A+v_B)$. Now $\sum_{i=0}^T v_A(\widehat\xi_i,\widehat\xi_t)$ and $\sum_{i=0}^T v_B(\widehat\xi_i,\widehat\xi_t)$ are both between zero and  $\sum_{i=0}^T \bigl( x^{\prime\star}(\widehat\xi_i,\widehat\xi_t)-x^{\star}(\widehat\xi_i,\widehat\xi_t) \bigr)$ and hence cannot have opposite signs for any $t=1,\ldots,T$. Thus, both perturbations $v_A$ and $v_B$ act in the same direction on each additively separable concave log term appearing in $\pi$, and therefore satisfy a diminishing returns property. Together with the linearity of the distance terms, it follows that
\[
\pi(x^\star+v_A)-\pi(x^\star)
\ge 
\pi(x^\star+v_A+v_B)-\pi(x^\star+v_B).
\]
Combining this with the inequalities from \eqref{equation:from-optimality} and using $x^{\prime\star} = x^{\star} + v_A + v_B$ gives
\[
0 \ge \pi(x^\star+v_A)-\pi(x^\star)
\ge \pi(x^\star+v_A+v_B)-\pi(x^\star+v_B)
\ge 0,
\]
so the inequalities are tight and $\pi(x^\star+v_A)=\pi(x^\star)$. 

We now suppose that $v(\widehat{\xi}_s,\widehat{\xi}_{T+1}) >0$ and will find a contradiction. By construction, this implies that $v_A(\widehat{\xi}_s,\widehat{\xi}_{T+1}) > 0$. Consider the convex combination $\frac{1}{2}x^\star + \frac{1}{2}(x^\star+v_A) = x^\star+ \frac{1}{2}v_A$  which is also feasible for the problem with the unswapped sequence $\sigma$. Conservation of flow and strict concavity of the objective function $\pi$ in $\sum_{i=0}^T x(\widehat{\xi}_i,\widehat{\xi_s})$ (which appears in the $s$-th log term) implies that $\pi(x^\star+\frac{1}{2}v_A) >\pi(x^\star+v_A)=\pi(x^\star)$, contradicting the optimality of $x^\star$. This establishes $v(\widehat\xi_s,\widehat\xi_{T+1}) \leq 0$, i.e., that \(x^{\prime\star}(\widehat{\xi}_j,\widehat{\xi}_{T+1}) \le x^\star(\widehat{\xi}_j,\widehat{\xi}_{T+1})\).
\end{proof}

At first sight it seems that we can use this result to show that if a whole component is moved earlier in time, then in an optimal solution none of those observations can have their assigned terminal probability increased (by moving observations one at a time, starting from the earliest in the component). However, this argument can fail, since each move in this process may change the set of observations within the component.

\section{Numerical Results}\label{section:numerical-results}
\noindent In this section we evaluate the performance of the WPF method. We begin by demonstrating that the WPF method has value over other weighting approaches using a synthetic newsvendor problem. This setup allows us to estimate expected performances with high statistical precision. Our findings show that the dimensionality and complexity of the underlying distributions significantly influence how beneficial (or detrimental) WPF can be. We then test performance on price-forecasting and portfolio-optimization applications. Here 
observations are taken from multidimensional real-world data sets collected over long timescales, and we use a parameter-tuning scheme which could feasibly be employed in practice. 

As benchmarks, we compare to the performance of SAA, windowing, and simple exponential smoothing. Like WPF, these all define weighted empirical distributions on the 
observations. SAA sets 
\begin{equation*}
    p_T(1) =1/T,\;\ldots,\;p_T(T) = 1/T,
\end{equation*} 
windowing for window size $s \in [T]$ sets 
\begin{equation*}
    p_T(T) = 1/s,\;\ldots,\;p_T(T-s+1) = 1/s, \quad p_T(T-s) = 0,\;\ldots,\;p_T(1) = 0,
\end{equation*}
and (simple exponential) smoothing for decay rate $\alpha\in [0,1]$ sets
\begin{equation*}
    p_T(T)=\bar{\alpha},\;p_T(T-1)=\bar{\alpha}(1-\alpha),\;p_T(T-2)=\bar{\alpha}(1-\alpha)^2,\;\ldots,\;p_T(1)=\bar{\alpha}(1-\alpha)^T.
\end{equation*}
The value $\bar{\alpha}=\alpha/\bigl(1-(1-\alpha)^{T+1}\bigr)$ is chosen so that the probabilities sum to $1$.\footnote{For $\alpha=0$, we take the limiting form as $\alpha\to0$, corresponding to $p_T(T)=\cdots=p_T(1)=1/T$.}

Our evaluations of the WPF method use the exponential-conic reformulation of (\ref{problem:reduced-network-flow}), implemented in \texttt{Julia} \parencite{Julia} with \texttt{JuMP.jl} \parencite{JuMP} and \texttt{MathOptInterface.jl} \parencite{mathoptinterface}, and solved with \texttt{COPT} \parencite{COPT}. Distances are measured using the metrics induced by the $1$-, $2$-, and $\infty$-norms, which we refer to by $L_1$, $L_2$, and $L_{\infty}$, respectively.

\subsection{Multi-Modal and Multi-Dimensional Newsvendor}
Consider a newsvendor dealing in $m$ goods, having underage and overage cost vectors~$\costUnder$ and~${\costOver\in\reals_+^m}$, respectively, supplying a vector of random demands $\rxi\in\reals^m\sim\Prob$. The vector of optimal order quantities for each good solves 
\begin{equation*}
    \minimize_{x\in\Xi}~~\Expt_{\Prob}\Bigl[ \costUnder^\transpose (\rxi-x)_+ + \costOver ^\transpose (x-\rxi)_+ \Bigr].
\end{equation*}
For the nonstationary demand process, we use a mixture of $n$ multivariate normal distributions with modes that evolve independently over time. At each time $t\in[T]$ we suppose that 
\begin{equation*}
    \rxi_t\in\reals^m \sim \frac{1}{n} \sum_{i=1}^n \text{Normal}\bigl(\mu_t(i),\mathop{\mathrm{Diag}}(\sigma^2,\ldots,\sigma^2)\bigr),
\end{equation*}
where $\mu_t(i)\in\reals^m$ is the mean vector of mode $i$ at time $t$, and $\mathop{\mathrm{Diag}}(\sigma^2,\ldots,\sigma^2)$ is an $m$-by-$m$ diagonal matrix with entries $\sigma^2$. For each mode $i$ we further suppose that 
\begin{equation*}
    \mu_{t+1}(i) = \mu_t(i) + \repsilon_t(i) ~~\where~~\repsilon_t(i) \in\reals^m \sim \text{Normal}\bigl((0,\ldots,0),\mathop{\mathrm{Diag}}(\rho^2,\ldots,\rho^2)\bigr).
\end{equation*}
In this way the scalar $\rho$ parametrises the extent of nonstationarity. For simplicity we use isotropic normal distributions and do not vary the covariance matrices over time.

In the case that there is a single mode then we have a problem for which the Kalman filter is appropriate; there is a random walk carried out by a state (the mean $\mu_t$), with observations occurring with a normal measurement error. The usual Kalman filter recursions then give an optimal estimate for the current mean. This can be shown, after some initial transient behaviour, to be given by an exponentially smoothed weighting of the observations. Thus, it is not surprising that the smoothing performs very well when there is a single mode, as is shown in Table~\ref{table: varying m and n} below. There is, however, a significant effect that comes into play when the distribution, rather than the mean, is needed for the stochastic optimization. When taking a weighted empirical distribution the mean will be essentially the optimal estimator for the mean of the underlying distribution, but the variance of the weighted empirical distribution will be consistently larger than the variance of the true underlying distribution.

To illustrate the behaviour of the WPF method, Figure~\ref{figure:radio-pulsar} shows an instance of what happens when there are two modes and we work in a single dimension. This is for one set of realisations of the process and a particular shift penalty.
\begin{figure}[H]
    \centering
    \hspace*{-0.4cm}\includegraphics[width=343pt]{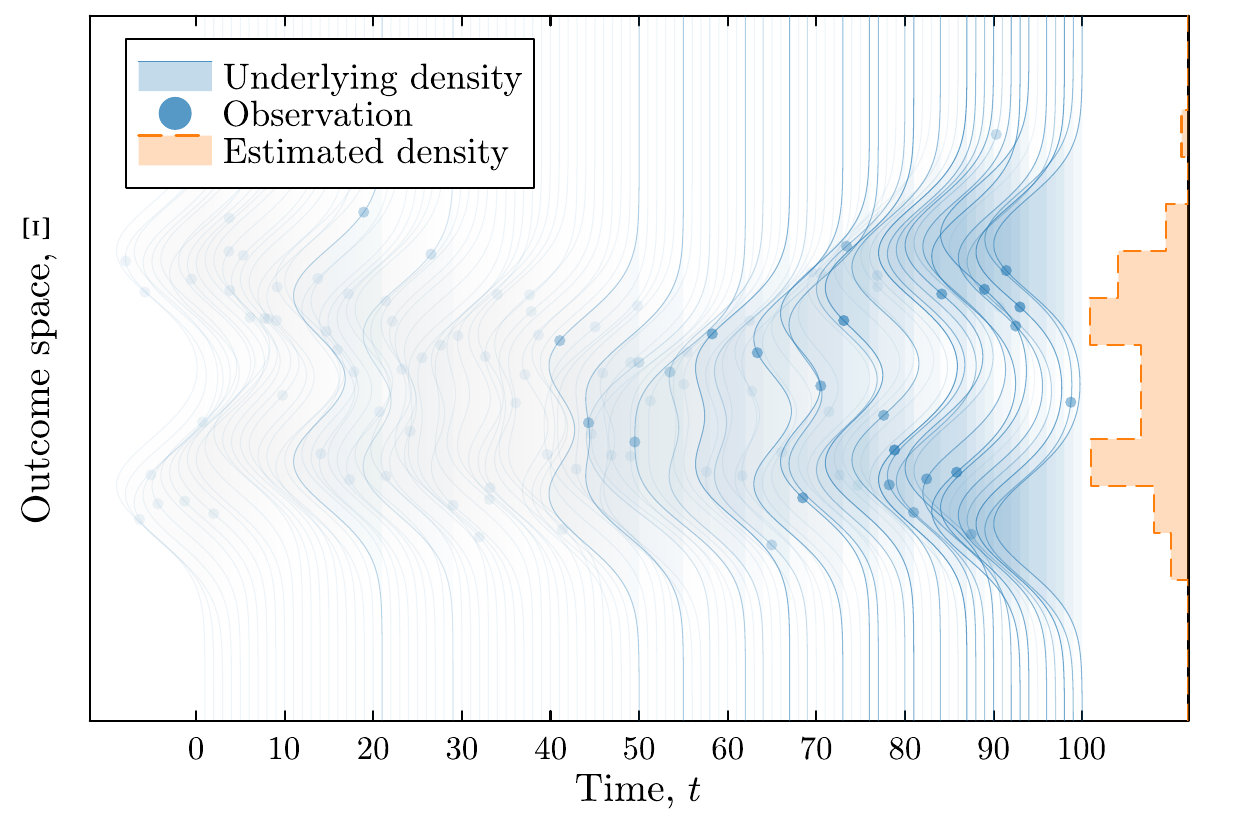}
    \caption{\textbf{Underlying Distributions, Observations, and Terminal WPF Estimate.} Opacities reflect probabilities in the terminal estimate. Here $\sigma = 20$, $\rho = 4$, and $\lambda = 10$.}
    \label{figure:radio-pulsar}
\end{figure}

Next we consider an example where we use different estimation methods to determine the optimal newsvendor order quantities given observations from the nonstationary process~$\rxi_t$. In all our examples we have an initial set of modes given by $\mu_1(i) = i(100,\ldots,100)$, $i\in[n]$, distribution parameters $\sigma = 20$ and $\rho = 15$, and cost vectors $\costUnder = (4,\ldots,4)$ and $\costOver = (1,\ldots,1)$. We also use a history length of $T=100$. In comparing different methods we take the average performance over $1000$ realisations and tune the values of the relevant parameters across the set of possible values shown in Table \ref{table:parameter-ranges}. For scalars $a \leq b$ and a positive integer $r$, we use the notation $\linRange(a,b;r)$ for the set of $r$ values forming an equally spaced sequence between (inclusive of) $a$ and $b$. Further, for scalars $a \leq b \in (0,\infty)$ and a positive integer $r$, we use the notation $\logRange(a,b;r)$ for the set of $r$ values forming a geometric sequence between (inclusive of) $a$ and $b$. This turns out to be the natural choice for parametrising windowing and smoothing.

\begin{table}[H]
\centering
\caption{\textbf{Newsvendor-Ordering Parameter Ranges}}
\label{table:parameter-ranges}
\centerline{\begin{tabular}{r c}
\toprule
\midrule
\addlinespace
    Window size & $s \in \lceil\logRange(1,100;30)\rceil$ \\
    Decay rate & $\alpha \in \{0\}\cup\logRange(10^{-4},10^{0};30)$ \\
    Shift penalty & $ \lambda \in \{0\}\bigcup_{i=-3}^{-1} \linRange(10^i,10^{i+1};10)\cup\{\infty\}$\\
\addlinespace
\bottomrule
\end{tabular}}
\end{table}

First we consider an example with two goods and a demand distribution with three modes. The results are shown in Table \ref{table: WPF comparisons}. We use three different distance norms from which to derive the Wasserstein distances. In this case all do well, with a distance based on the $L_{\infty}$ norm doing best. In the other numerical examples we discuss below, it turns out that switching from e.g.~$L_1$~to~$L_2$ can make a big difference, but here it does not matter so much.

We find that both windowing and smoothing give a significant improvement over the basic SAA approach that ignores the nonstationarity. But all versions of the WPF approach are better than these other methods, with cost differences of more than $\approx 3$\,\% between WPF and Smoothing.

\begin{table}[H]
\caption{\textbf{Newsvendor-Ordering Performance.} Standard errors ($\pm$ terms) reported in `Difference from SAA' row computed using common random numbers.}\label{table: WPF comparisons}
\centering
\centerline{\begin{tabular}{r cccccc}
\toprule
\midrule     
      & \multirow{2.25}{*}{\hphantom{+}SAA} & \multirow{2.25}{*}{\hphantom{+}Windowing} & \multirow{2.25}{*}{\hphantom{+}Smoothing} & \multicolumn{3}{c}{WPF  with metric} \\\cmidrule(lr){5-7}
      & & & & $\hphantom{+}L_1$ & $\hphantom{+}L_2$ & $\hphantom{+}L_\infty$ \\
\midrule
\addlinespace

\makecell[r]{Ex-post optimal\\expected cost} 
        & $\hphantom{+}427.2$
        & $\hphantom{+}385.6$
        & $\hphantom{+}377.6$
        & $\hphantom{+}368.8$
        & $\hphantom{+}368.0$
        & $\hphantom{+}368.1$ \\
\addlinespace
\midrule
\addlinespace
\makecell[r]{Difference\\from SAA (\%)}
        & $ $
        & $-9.7 \pm 0.5$
        & $-11.6 \pm 0.5$
        & $-13.7 \pm 0.5$
        & $-13.9 \pm 0.5$
        & $-13.8 \pm 0.5$ \\
        
\addlinespace
\bottomrule
\end{tabular}}
\end{table}

Next we consider a range of different dimensions $m$ and number of modes $n$ for the problem. The results are shown in Table~\ref{table: varying m and n}. We can see that good performance of the WPF method is achieved when there are a greater number of modes and the dimension is greater than one. Roughly speaking, the more complicated the nonstationary behaviour the better WPF will perform in comparison to other methods.

\begin{table}[H]
\centering
\caption{\textbf{Percentage Cost Difference from Smoothing for WPF with $\bm{L_1}$ Metric}}
\label{table: varying m and n}
\begin{tabular}{r ccccc}
\toprule
\midrule
    \multirow{2.25}{*}{\makecell{Dimensions}} & \multicolumn{4}{c}{Underlying modes}\\\cmidrule(lr){2-5}
     & $\hphantom{+}1$ & $\hphantom{+}2$ & $\hphantom{+}3$ & $\hphantom{+}4$ \\
\midrule
\addlinespace

    $1$ & $\hphantom{+}12.8 \pm 1.3$ & $\hphantom{+}0.2 \pm 0.7$ & $-0.9 \pm 0.5$ & $\hphantom{+}0.5 \pm 0.3$ \\
    $2$ & $\hphantom{+}11.7 \pm 0.9$ & $-1.0 \pm 0.6$ & $-2.3 \pm 0.3$ & $-1.0 \pm 0.2$ \\
    $3$ & $\hphantom{+}11.4 \pm 0.8$ & $-1.4 \pm 0.6$ & $-3.8 \pm 0.3$ & $-1.3 \pm 0.2$ \\
    $4$ & $\hphantom{+}11.3 \pm 0.7$ & $-2.4 \pm 0.6$ & $-3.7 \pm 0.3$ & $-1.7 \pm 0.2$ \\
    
\addlinespace
\bottomrule
\end{tabular}
\end{table}

\subsection{Price Forecasting}
In the remainder of this section we adopt a training-and-testing approach to evaluate performance, splitting time-series data into a training set (the first $70\,\%$ of the data) and a testing set (the remaining $30\,\%$ of the data). In the training phase decisions are not made during a warm-up period while initial information (observations) are collected. After this, decisions are made sequentially, incurring costs based on the subsequent observation from within the training set. These costs inform parameter tuning in the testing phase, where at each time step the method selects the parameter that minimizes the total cost over a fixed parameter tuning window.

In the testing phase (which uses the testing set), decisions continue to be made using all prior observations, with parameters updated dynamically based on performance in the parameter tuning window. This window introduces a trade-off: a longer window yields more stable parameter estimates, while a shorter window responds more quickly to changes in time-series dynamics. In our evaluations we found that a two year tuning window worked well for WPF, and that varying this window had little effect on the benchmarks. Note that in this training-and-testing set up, decision making and parameter tuning rely solely on data collected in the past, i.e., the tests are consistent with how this method could be applied in practice \parencite{tashman-2000-out-of-sample-forecasting}.

The international market prices of dairy commodities are influenced by factors such as global economic conditions and consumer demand, both of which evolve over time. Forecasts of such prices inform the operational decision making of dairy companies around the world. Figure~\ref{figure:dairy-prices-data} graphs fourteen years of monthly Global Dairy Trade (GDT) prices for Anhydrous Milk Fat (AMF), Butter (BUT), Butter Milk Powder (BMP), Skim Milk Powder (SMP), and Whole Milk Powder (WMP). The GDT price is the average per-unit price of each product sold at international commodity auctions in a given month.
\begin{figure}[H]
    \centerline{
    \hspace*{1.6cm}\includegraphics[width=399pt]{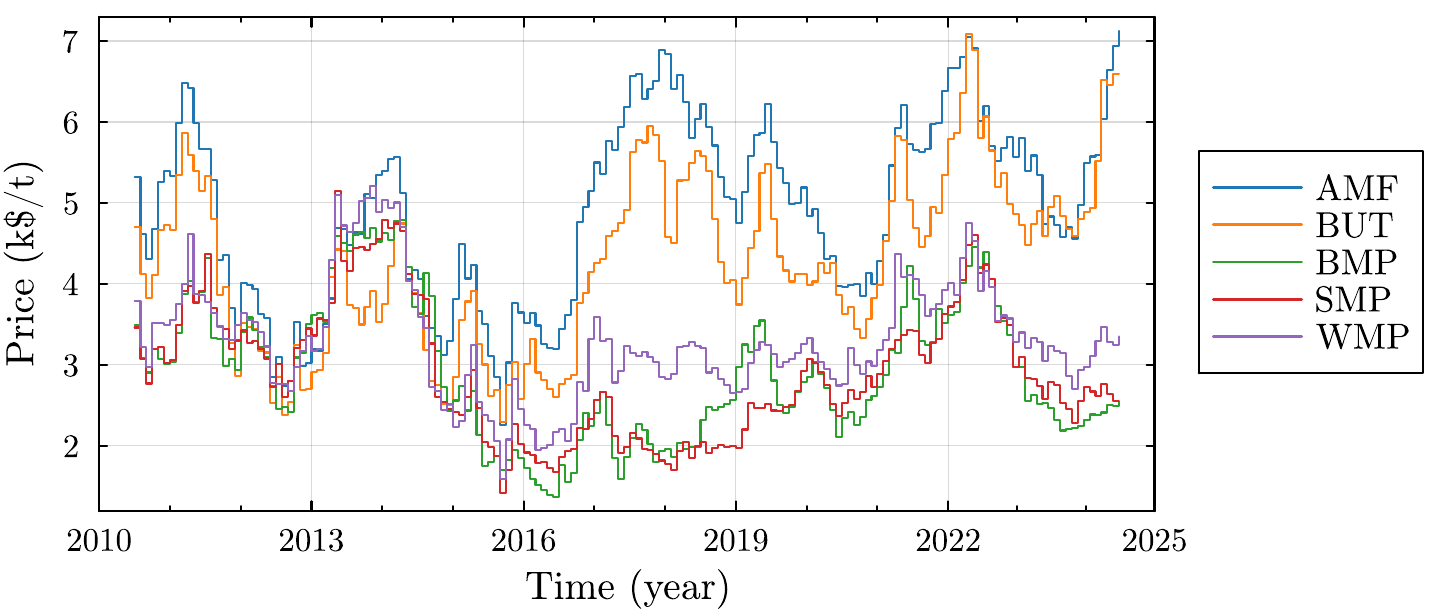}}
    \caption{\textbf{Historical GDT Prices.} Monthly observations; June 2010--May 2024.\protect\footnotemark}
    \label{figure:dairy-prices-data}
\end{figure}
\footnotetext{Data retrieved from \parencite{dairy-forecasting}.}
\noindent Figure~\ref{figure:dairy-prices-data} shows that the GDT prices of each product are correlated. In particular, the prices of the fat products (AMF and BUT) and the milk powders (BMP, SMP, and WMP) are coupled throughout the fourteen-year period, with a change in behaviour around $2016$.

To forecast future GDT prices, we use a weighted linear regression on the previous log prices. For historical log-prices~$\ell_1,\ldots,\ell_T \in \reals^m$ observed at times $t=1,\ldots,T$, assigning the pairs~$(\ell_1,\ell_2),\ldots,(\ell_{T-1},\ell_T)$ probabilities~$p_1,\ldots,p_{T-1}$, we forecast the log price at time~$T+1$ as~$\mu^{\star} + A^{\star} \ell_T$, where $\mu^{\star}$ and $A^{\star}$ solve
\begin{equation*}
    \minimize_{\mu\in\reals^{m},A\in\reals^{m\times m}} ~~ \sum_{t=1}^{T-1} p_t\bigl\lVert\ell_{t+1}-(\mu+A\ell_t)\bigr\rVert^2_2.
\end{equation*}
(Note that this has an analytic solution; see, e.g., \parencite[Section~11.5]{seber-lee:Linear-Regression-Analysis}.) When the next log-price~$\ell_{T+1}$ is realised, we incur the cost $\bigl\lVert\ell_{T+1}-(\mu^{\star}+A^{\star}\ell_T)\bigr\rVert^2_2$. The weighted regression approach here allows adaptation to shifting price dynamics.

Table~\ref{table:price-forecasting-parameter-ranges} presents the parameter ranges used when tuning, and Table~\ref{table:dairy-prices-test} presents the price-forecasting performance of the WPF method and the benchmarks.

\begin{table}[H]
\centering
\caption{\textbf{Price-Forecasting Parameter Ranges}}
\label{table:price-forecasting-parameter-ranges}
\centerline{\begin{tabular}{r c}
\toprule
\midrule
\addlinespace
    Window size & $s \in \lceil\logRange(10,14\cdot 12;30)\rceil$\\
    Decay rate & $\alpha \in \{0\}\cup\logRange(10^{-4},0.9\cdot 10^{0};30)$ \\
    Shift penalty & $\lambda \in \bigcup_{i=1}^3 \linRange(10^i,10^{i+1};10)\cup\{\infty\}$\\
\addlinespace
\bottomrule
\end{tabular}}
\end{table}

\begin{table}[H]
\caption{\textbf{Price-Forecasting Performance.} Standard errors ($\pm$ terms) reported in `Difference from SAA' row computed using common random numbers.}\label{table:dairy-prices-test}
\centering
\centerline{\begin{tabular}{r cccccc}
\toprule
\midrule     
      & \multirow{2.25}{*}{\hphantom{+} SAA} & \multirow{2.25}{*}{\hphantom{+} Windowing} & \multirow{2.25}{*}{\hphantom{+} Smoothing} & \multicolumn{3}{c}{WPF with metric} \\\cmidrule(lr){5-7}
      & & & & $\hphantom{+}L_1$ & $\hphantom{+}L_2$ & $\hphantom{+}L_\infty$ \\
\midrule
\addlinespace
\makecell[r]{Average\\testing cost} 
        & $\hphantom{+}0.0261$
        & $\hphantom{+}0.0277$
        & $\hphantom{+}0.0267$
        & $\hphantom{+}0.0242$
        & $\hphantom{+}0.0280$
        & $\hphantom{+}0.0251$ \\
\addlinespace
\midrule
\addlinespace
\makecell[r]{Difference\\from SAA (\%)}
        & $ $
        & $\hphantom{+}5.8\pm 5.7$
        & $\hphantom{+}2.1\pm 10.6$
        & $-7.5\pm 5.7$
        & $\hphantom{+}7.0\pm 4.5$
        & $-3.9\pm 8.3$ \\
\addlinespace
\bottomrule
\end{tabular}}
\end{table}
\noindent Table~\ref{table:dairy-prices-test} shows that the WPF method with the $L_1$ metric provides $\approx 8\,\%$ better forecasts than SAA, while windowing and smoothing perform similarly to SAA. 

Figure~\ref{figure:dairy-prices-WPF-parameter-costs} graphs average costs of the WPF method for different shift penalties. There is a large decrease in costs as the penalty approaches an intermediate value and the solution begins to utilise the information contained in historical observations, and then a gradual rise and stabilisation in costs as the solution approaches SAA.
\begin{figure}[H]
    \centering
    \hspace{-1.4cm}\includegraphics[width=281pt]{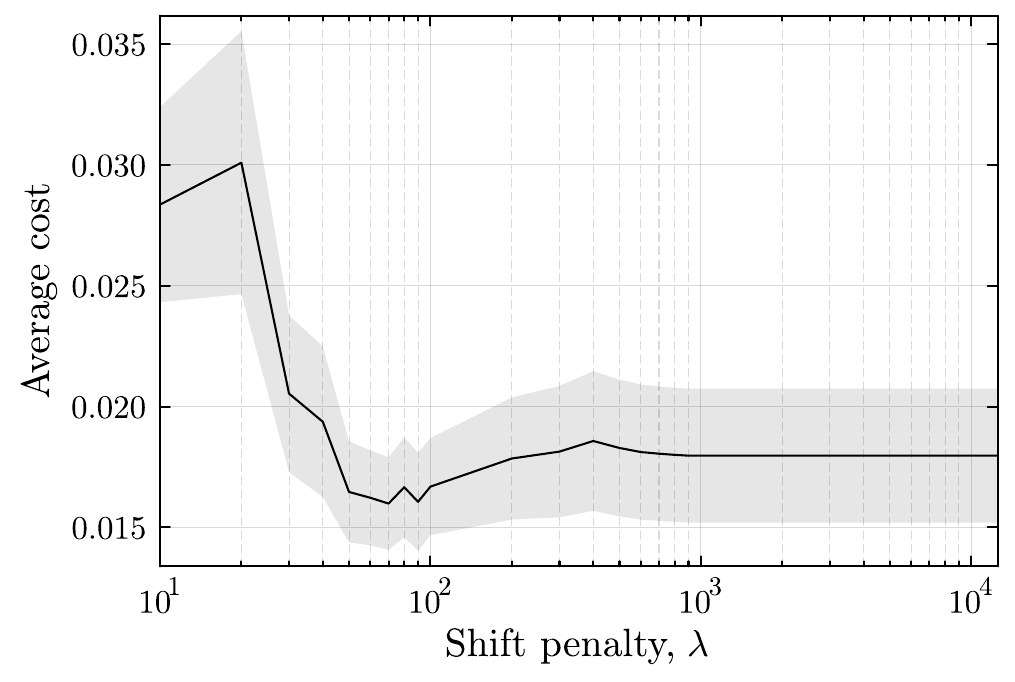}
    \caption{\textbf{Price-Forecasting Parameter Tuning for WPF with $\bm{L_1}$ Metric.} Average cost within the two-year parameter-tuning window. Band presents standard-error range.}
    \label{figure:dairy-prices-WPF-parameter-costs}
\end{figure}

Figure~\ref{figure:dairy-prices-WPF-distribution} graphs the probabilities assigned to the sequence of historical observations from Figure~\ref{figure:dairy-prices-data} by WPF when price forecasting. A careful comparison of Figure~\ref{figure:dairy-prices-WPF-distribution} to Figure~\ref{figure:dairy-prices-data} shows that the WPF method generally only assigns probabilities to 
observations in which the prices of the two fat products are similar and the prices of each milk-powder product are similar, as is the case for the most recent price~dynamics.

\begin{figure}[H]
    \centering
    \hspace*{-1cm}\includegraphics[width=329pt]{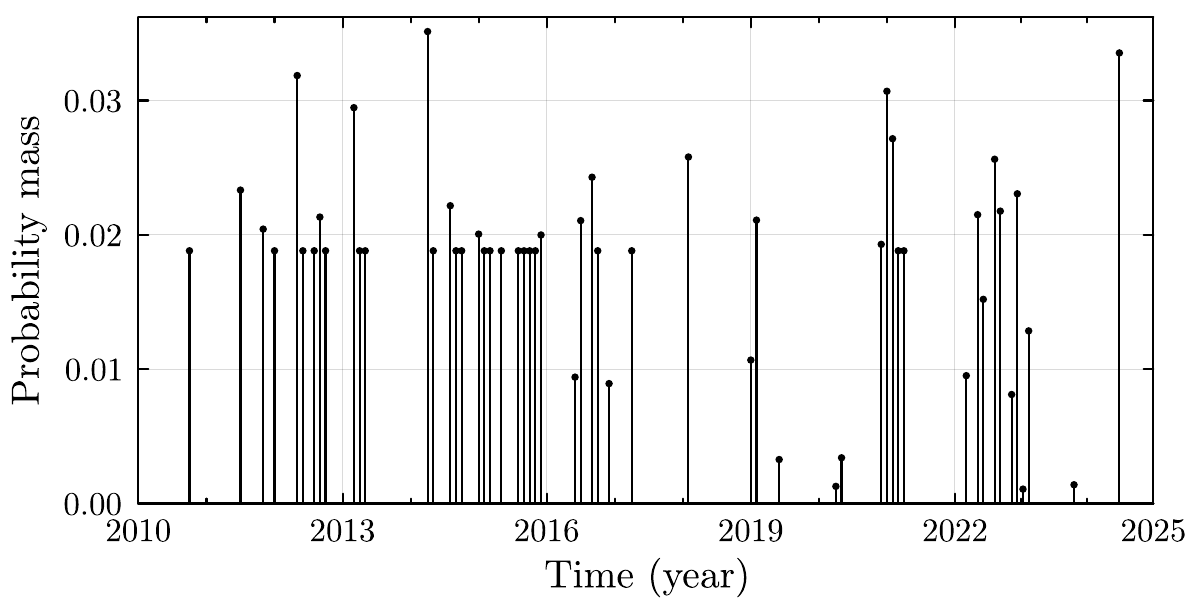}
    \caption{\textbf{Price-Forecasting Probabilities for WPF with $\bm{L_1}$ Metric.} Shift penalty set at the minimum from Figure~\ref{figure:dairy-prices-WPF-parameter-costs}.}
    \label{figure:dairy-prices-WPF-distribution}
\end{figure}

\subsection{Portfolio Optimization}
Figure~\ref{figure:stock-returns-data} graphs ten years of stock values for Boeing Co.\ (BA), Berkshire Hathaway Class~B~(BRK.B), Goldman Sachs Group Inc.\ (GS), Johnson \& Johnson (JNJ), JPMorgan Chase~\&~Co.~(JPM), Coca-Cola Co.\ (KO), McDonald’s Corp.\ (MCD), Pfizer Inc.\ (PFE), Walmart Inc.\ (WMT), and Exxon Mobil Corp.\ (XOM).
\begin{figure}[H]
    \centerline{
    \hspace*{1.8cm}\includegraphics[width=399pt]{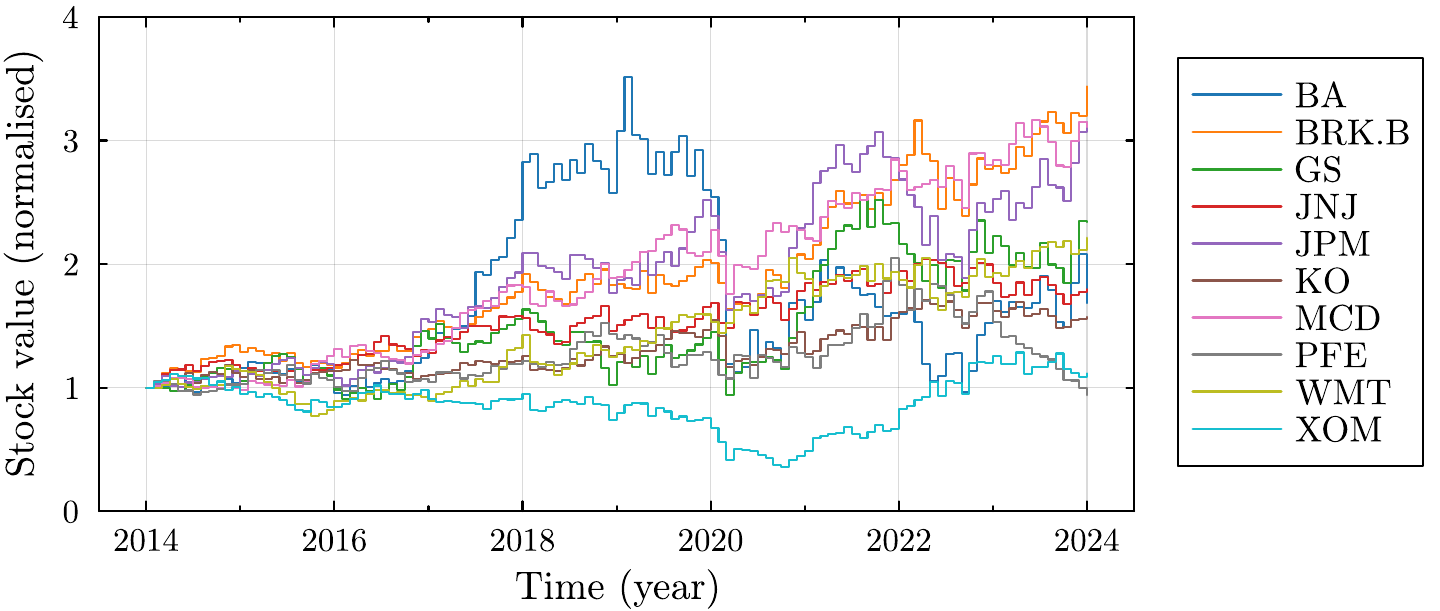}}
    \caption{\textbf{Historical Stock Values.} Monthly observations; Jan.\ 2014--Dec.\ 2024.\protect\footnotemark}
    \label{figure:stock-returns-data}
\end{figure}
\footnotetext{Data retrieved using the \texttt{STOCKHISTORY} function in \texttt{Microsoft Excel}.}

\noindent Figure~\ref{figure:stock-returns-data} shows that the returns of each stock are correlated, with market shocks affecting most stocks simultaneously. In particular, there are significant shocks in early 2020 and in mid-2022.

We consider a risk-averse portfolio-optimization problem. For $m$ stocks with portfolio weights~$x\in\reals^m_+$ and random returns~$\rxi\in\reals^m$, the risk-neutral portfolio return is the random variable~$x^{\transpose}\rxi$. Adjusting for the conditional value at risk of the portfolio, the problem is
\begin{equation*}
    \minimize_{x\in\reals_+^m} ~~ (1-\rho) \Expt\bigl[-x^{\transpose}\rxi \bigr] + \rho\,\mathrm{CVaR}_{\beta}\bigl[-x^{\transpose}\rxi\bigr] ~~ \subjectTo ~~ \sum_{i=1}^m x_i=1.
\end{equation*}
Here $\mathrm{CVaR}_{\beta}$ is the conditional-value-at-risk operator at percentile $\beta\in[0,1)$, and $\rho\in[0,1]$ parametrises the level of risk aversion. 
(Note that $\beta$ is associated with the worst $(1-\beta)\,\%$ of returns, so $\beta=0$ recovers the risk-neutral case.) For $T$ historical returns $\xi_1,\ldots,\xi_T\in\reals^m$ assigned probabilities $p_1,\ldots,p_T$, using the Rockafellar--Uryasev formula gives
\begin{alignat*}{2}
    & \minimize_{x\in\reals_+^m,\tau\in\reals} && \quad (1-\rho) \sum_{t=1}^T p_t\Bigl(-x^{\transpose}\xi_t\Bigr) + \rho \sum_{t=1}^T p_t\Bigl(\tau + \tfrac{1}{1-\beta}\max\bigl\{-x^{\transpose}\xi_t -\tau,0\bigr\}\Bigr)\\ 
    & \text{subject to} && \quad \sum_{i=1}^m x_i=1. \notag
\end{alignat*}
This can further be expressed as a linear program via an epigraphical reformulation.

In our evaluations we set $\rho = 0.9$ and $\beta=95\,\%$. Table~\ref{table:portfolio-parameter-ranges} presents the parameter ranges used when tuning, and Table~\ref{table:stock-returns-test} presents the  performance of different estimation methods.
\begin{table}[H]
\centering
\caption{\textbf{Portfolio-Optimization Parameter Ranges}}
\label{table:portfolio-parameter-ranges}
\centerline{\begin{tabular}{r c}
\toprule
\midrule
\addlinespace
    Window size & $s \in \lceil\logRange(1,10\cdot 12;30)\rceil$\\
    Decay rate & $\alpha \in \{0\}\cup\logRange(10^{-4},10^0;30)$ \\
    Shift penalty & $\lambda \in \{0\}\cup\bigcup_{i=0}^2 \linRange(10^i,10^{i+1};10)\cup\{\infty\}$\\
\addlinespace
\bottomrule
\end{tabular}}
\end{table}

\begin{table}[H]
\centering
\caption{\textbf{Portfolio-Optimization Performance.} Standard errors ($\pm$ terms) reported in `Difference from SAA' row computed using common random numbers.}
\label{table:stock-returns-test}
\centerline{\begin{tabular}{r cccccc}
\toprule
\midrule        
      & \multirow{2.25}{*}{\hphantom{+}SAA} & \multirow{2.25}{*}{\hphantom{+}Windowing} & \multirow{2.25}{*}{\hphantom{+}Smoothing} & \multicolumn{3}{c}{WPF with metric} \\\cmidrule(lr){5-7}
      & & & & $\hphantom{+}L_1$ & $\hphantom{+}L_2$ & $\hphantom{+}L_\infty$ \\
\midrule
\addlinespace
\makecell[r]{Risk-adjusted\\average testing cost}
        & $\hphantom{+}0.0707$
        & $\hphantom{+}0.0930$
        & $\hphantom{+}0.0896$
        & $\hphantom{+}0.0630$
        & $\hphantom{+}0.0772$
        & $\hphantom{+}0.0644$ \\
\addlinespace
\midrule
\addlinespace
\makecell[r]{Difference\\from SAA (\%)}
        & $ $
        & $\hphantom{+}31.6\pm 9.4$
        & $\hphantom{+}26.8\pm 8.0$
        & $-10.8\pm 5.8$
        & $\hphantom{+}9.2\pm 5.3$
        & $-8.8\pm 5.6$ \\
\addlinespace
\bottomrule
\end{tabular}}
\end{table}
\noindent Table~\ref{table:stock-returns-test} shows that the WPF method with the $L_1$ metric provides $\approx 11\,\%$ better returns than SAA, while windowing and smoothing both perform worse than SAA. 

Figure~\ref{figure:stock-returns-WPF-parameter-costs} graphs average costs of the WPF method for different shift penalties. In contrast to Figure~\ref{figure:dairy-prices-WPF-parameter-costs}, there are two decreases in costs for intermediate penalties. 
\begin{figure}[H]
    \centering
    \hspace{-1.3cm}\includegraphics[width=281pt]{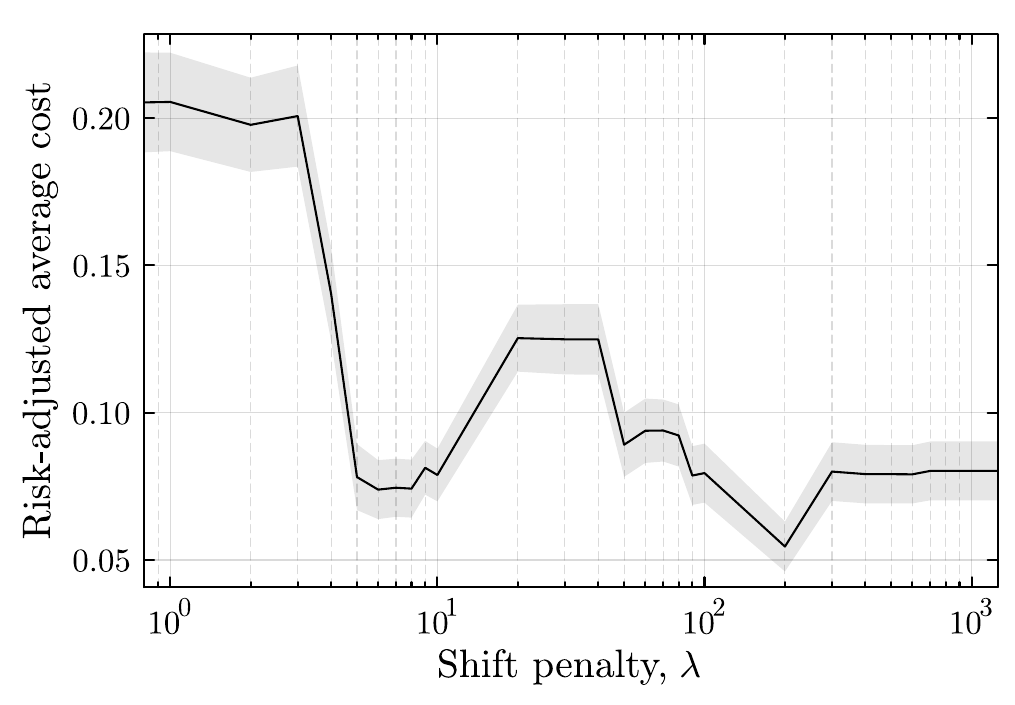}
    \caption{\textbf{Portfolio-optimization Parameter Tuning for WPF with $\bm{L_1}$ Metric.} Risk-adjusted average cost within the two-year parameter-tuning window. Band presents standard-error range.}
    \label{figure:stock-returns-WPF-parameter-costs}
\end{figure}


Figure~\ref{figure:stock-returns-WPF-distribution} graphs the probabilities assigned to the sequence of 
observations from Figure~\ref{figure:stock-returns-data} by WPF when optimizing the portfolio. In comparison to Figure~\ref{figure:dairy-prices-WPF-distribution}, Figure~\ref{figure:stock-returns-WPF-distribution} features more observations with the same probabilities. This is evidence of a larger number of singleton components in the optimal network-flow. 
\begin{figure}[H]
    \centering
    \hspace*{-1.4cm}\includegraphics[width=329pt]{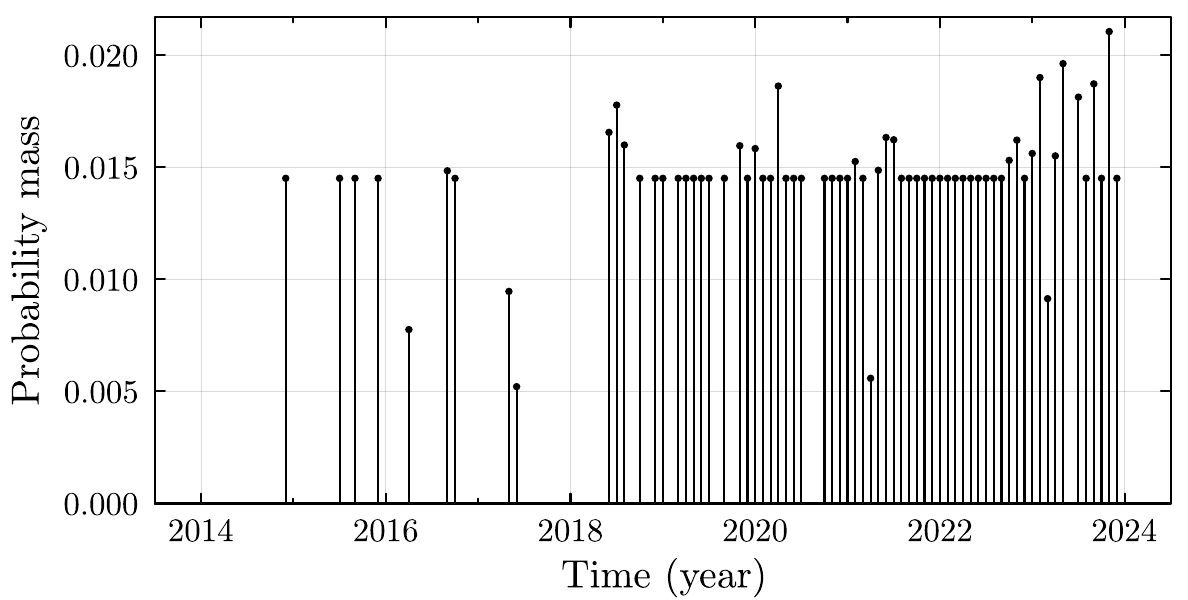}
    \caption{\textbf{Portfolio-Optimization Probabilities for WPF with $\bm{L_1}$ Metric.} Shift penalty set at the minimum from Figure~\ref{figure:stock-returns-WPF-parameter-costs}.}
    \label{figure:stock-returns-WPF-distribution}
\end{figure}



\section{Discussion}\label{section:conclusion}
We have shown how  the maximum-likelihood formulation with Wasserstein-distance regularization gives a problem that can be reformulated to have a simple network-flow structure. Using synthetic data we show how this method performs well in cases where the underlying probability structures involve greater complexity -- in our experiments we consider multimodal distributions in which the modes move independently. In two examples with nonstationary, real-world data, we demonstrate how effective this approach can be in practice. This is achieved with the $L_1$ metric, which here performs better than the Euclidean metric. From a maximum-likelihood perspective, if the probability of a change in the distribution across multiple dimensions is simply the product of the probabilities of the change across each individual dimension, then after taking logarithms we get a framework in which the $L_1$ metric appears naturally. Even though in our two example data sets we do not expect the required dimension independence to hold, the results suggest that the $L_1$ metric may be a better choice than alternatives.

The flexibility of the WPF method with regard to different distance metrics is an advantage. 
We have already seen that the $L_1$ metric {can be} much more effective than the Euclidean metric when dealing with multidimensional data. Another option is to adjust an existing metric to penalise any movement -- thus we introduce a new parameter 
$\delta > 0$ and define the adjusted distance by
$$
d^{\prime}(\xi,\zeta) \defeq \begin{cases}
d(\xi,\zeta)+\delta & \text{if}~\xi \neq \zeta,\\
0 & \text{if}~\xi = \zeta.
\end{cases}
$$
As long as $d$ is a lower-semicontinuous metric, it can be seen that $d^{\prime}$ is a lower-semicontinuous metric as well. This change introduces a sort of `stickiness', with more of the distribution being left unchanged from one period to the next. We found that this can produce marginal improvements in our numerical examples.

The WPF method results in an assignment of nonzero probabilities to a subset of the previous observations as an estimate of the final probability distribution. Depending on the value of the shift penalty, the number of observations included in this set may be quite small. The WPF method thus provides a suitable methodology for scenario reduction in the presence of nonstationarity. Existing approaches to scenario reduction focus on selecting a subset of data points that give a good representation of the overall distribution, where a Wasserstein distance can be used to measure the fit \parencite{rujeerapaiboon2022scenario}. The need to make a small selection of scenarios from the complete historical record is driven by computational necessity in areas like energy planning \parencite{park2019comparing}.

The model we use has a single penalty parameter $\lambda$, corresponding to a situation in which the likelihood of changes in the underlying distribution is constant over time. So although we have nonstationarity, the characteristics of  this nonstationarity do not vary over time. For example, our model could be applied in a case where the underlying distribution in each time period is chosen from a finite set of possible distributions and movement between these candidate distributions follows a Markov chain. The Wasserstein penalty would then be appropriate if the probability of moving from one state to another depends on the Wasserstein distance between the corresponding distributions. Then our assumption of a single value for $\lambda$ corresponds to a fixed transition matrix. In our empirical examples we see variations in $\lambda$ over the period of the time series, which suggests that \emph{higher-order} nonstationarity -- where the dynamics of the nonstationarity vary over time -- may be important. We have not attempted a systematic exploration of this, but it seems that explicitly discarding older observations could be  effective in these circumstances. 

WPF solutions exhibit certain behaviours that generally hold, but may break down in specific cases. 
For example, it is generally true that larger values of $\lambda$ lead to more of the observations appearing in the final distribution with nonzero probabilities. At the extreme, all of the observations appear as $\lambda \to \infty$. However, in some situations we can find observations dropping out when $\lambda$ is increased.  (An example occurs with observations $\{6.41, 6.4, 5.89, 5.69, 5.13, 4.5695 \}$: as $\lambda$ increases from $2.7$ to $3$,  the fifth observation, $5.13$, drops out.)

The WPF method estimates a distribution in every time period. Except for the first and last time period, there will generally be many different distributions that can be chosen. However, as a consequence of the argument in the proof of Lemma \ref{lemma:discrete-support} we know that $\mathbb{P}_t(\{\widehat{\xi}_t\})$, the probability on the observation $\widehat{\xi}_t$ at time $t$, is fixed at optimality. Now consider what happens to this as the total number of observations $T$ increases. As new observations are added, there may be some adjustment to $\mathbb{P}_t(\{\widehat{\xi}_t\})$, but for $T$ sufficiently large this adjustment will be  negligible or nonexistent. The behaviour of the optimal solution around time period $t$ should not be influenced by observations far in the future, and thus $\mathbb{P}_t(\{\widehat{\xi}_t\})$ for fixed $t$ may approach a limit as $T \to \infty$. This is a type of stability result which would be desirable. However it may fail to hold in specific cases. We give an example of such a failure in Appendix~\ref{appendix:example-instability}. (Note that this instability only occurs in very specific cases that have zero probability in practice.)

Finally we note that there is a possible extension of the WPF method to include a penalty for changes in distribution over the recent past, rather than just the current distribution. For example we could include a distance over two time steps in the objective function of WPF so that we get:
$$
\maximize_{\mathbb{P}_1,\ldots,\mathbb{P}_T} \quad \sum_{t=1}^{T} \log\bigl(\mathbb{P}_t(\{\widehat{\xi}_t\})\bigr) - \lambda_1  \sum_{t=1}^{T-1} W(\mathbb{P}_t, \mathbb{P}_{t+1})- \lambda_2 \sum_{t=1}^{T-2} W(\mathbb{P}_t, \mathbb{P}_{t+2}).
$$
However, this problem is much harder to solve and the key property of being able to restrict attention to distributions supported only on the observations will fail in this case. An example where this occurs is when there are four observations in $\reals^2$:   $\widehat{\xi}_1=(4,1)$, $\widehat{\xi}_2=(10,2)$, $\widehat{\xi}_3=(2,2)$, $\widehat{\xi}_4=(4,3)$, and we use Euclidean distances. Then if $\lambda_1 = \lambda_2 = 0.5$ we find that an optimal solution has positive probability in period 2 on the point $(3.42265, 2)$, which is the geometric median of the points $\widehat{\xi}_1$, $\widehat{\xi}_3$ and $\widehat{\xi}_4$. 

\section*{Acknowledgements}
The authors would like to thank Oscar Dowson for the GDT price data. The authors would also like to thank Andrew Philpott and Wolfram Wiesemann who have contributed helpful comments on this work. The second author was supported in part by the Claude McCarthy Fellowship.

\begin{appendices}
\section*{\textbf{\LARGE{Appendix}}}

\vspace{-0.2em}
\section{Example without Stability}
\label{appendix:example-instability}

Consider the following example. We take $\lambda=4$. The sequence of $\widehat{\xi}_t$ values are 
\[0,0,0,0,0,1,-1,-1,\underline{-1},1,1,1,-1,1,1,\underline{1},-1,-1,-1,1, \ldots ,-1,-1,\underline{-1},1,1,1,-1, \ldots .\]
This is defined as five zeros followed by a sequence of $1$ and $-1$ values split into alternating groups: within each group there is a repeated pattern, either three $-1$ values and then one $1$ value, or the reverse. It turns out that this pattern provides a sufficient local mixture of $1$s and $-1$s such that it is never worthwhile to include a path from $1$ to $-1$ or vice versa. The groups are of increasing length with a switch to the opposite group when the total number of observations of $-1$ are 50\,\% more than the number of observations of $1$, or vice versa. The first three times this switch occurs are at $t=9$, $t=16$ and $t=31$ (shown underlined in the list above). But it is clear that the reversals will happen infinitely often. Our aim is to show that as each observation is added to the sequence the optimal solution will be unstable, with no limiting value for the probability assigned to the first nonzero observation, $\widehat{\xi}_6$. 

We first establish that in a path decomposition of the optimal solution there is no value of $T>20$ for which a path with an arc from $1$ to $-1$ or vice versa can occur with nonzero mass. This arc is most likely to occur right at the end. Suppose that the final four observations are $\widehat{\xi}_{T-3}=1$, $\widehat{\xi}_{T-2}=\widehat{\xi}_{T-1}=\widehat{\xi}_{T}=-1$ and that the value of $\mathbb{P}_{T-3}(\{1\})=f$, with all of the other mass located at $-1$ (there will be none at $0$). The condition  needed to include a new path which has $x(T-2,T-1)>0$ is that the derivative along this new path is greater than $0$ since the alternative is to have $x(T-2,T+1)=f$. In other words, we require that $3/(1-f)-2\lambda>0$, i.e., $f>0.625$.  We will show for $T>20$ that both $f$ and $1-f$ are less than $0.625$ and hence that there is no path of this form. 

We may thus consider solutions which make use of just 4 paths (with probabilities $f_1,f_2,f_3,f_4$). $\cP_1$ goes from the source to the first observation at $1$, and then moves between successive observations at $1$. $\cP_2$ goes from the source to  $\widehat{\xi}_1$ at $0$ then stays at $0$ for the first five observations until it moves to the first observation at $1$, before following $\cP_1$. $\cP_3$ and $\cP_4$ match $\cP_1$ and $\cP_2$ but use observations at $-1$ rather than $1$. If there are a total of $n_1$ observations at $1$ and $n_{-1}$ observations at $-1$, then we can show that an optimal solution has $f_1=f_3=0$, $f_2=n_1/(n_1+n_{-1})$ and $f_3=n_{-1}/(n_1+n_{-1})$. This uses Lemma \ref{lemma:path-derivative}. We take $\pathConstant=5+n_1+n_{-1}-\lambda$ and note that (\ref{equation:path-constant}) holds with equality for paths $\cP_2$ and $\cP_3$, and that the inequality holds for $\cP_1$ and $\cP_4$ since $\lambda<5$. Now observe by construction we have $f_2$ varying between $1.5 n_{-1}/(2.5n_{-1})=0.6$ and $n_1/(2.5n_1)=0.4$. These are approximate expressions -- we switch between groups as soon as the ratio exceeds $3/2$, but the accuracy increases as $T$ gets larger. So this all works for sufficiently large $T$. (For smaller values of $T$ we may not meet the requirement for both $f_2$ and $f_3=1-f_2$ to be less than $0.625$.) With this example we have $f_2=0.25$ at $T=9$ and $f_2=0.636$ at $T=16$, but at $T=31$ we get $f_2=0.615$ which meets our requirement. 

We have therefore established that for sufficiently large $T$ the probability $\mathbb{P}_6(\{\widehat{\xi}_6\})$ will oscillate between approximately $0.4$ and $0.6$, and have no limiting value as successive observations are made~and~$T \rightarrow \infty$. 

\end{appendices}

\printbibliography

\end{document}